\documentclass[leqno]{amsart}
\usepackage{amsthm,amsmath,amssymb}
\usepackage{stmaryrd,mathrsfs}
\usepackage[T1]{fontenc}
\usepackage{esint}

\allowdisplaybreaks

\newcommand{\ud}[0]{\,\mathrm{d}}
\newcommand{\ceil}[1]{\lceil #1 \rceil}
\newcommand{\dist}[0]{\operatorname{dist}}

\newcommand{\abs}[1]{|#1|}
\newcommand{\Babs}[1]{\Big|#1\Big|}
\newcommand{\Norm}[2]{\|#1\|_{#2}}
\newcommand{\BNorm}[2]{\Big\|#1\Big\|_{#2}}
\newcommand{\pair}[2]{\langle #1,#2 \rangle}
\newcommand{\Bpair}[2]{\Big\langle #1,#2 \Big\rangle}
\newcommand{\ave}[1]{\langle #1\rangle}

\newcommand{\lspan}[0]{\operatorname{span}}
\newcommand{\BMO}[0]{\operatorname{BMO}}
\newcommand{\supp}[0]{\operatorname{supp}}

\newcommand{\R}{\mathbb{R}}
\newcommand{\C}{\mathbb{C}}
\newcommand{\N}{\mathbb{N}}
\newcommand{\Z}{\mathbb{Z}}

\newcommand{\prob}[0]{\mathbb{P}}
\newcommand{\Exp}[0]{\mathbb{E}}
\newcommand{\D}[0]{\mathbb{D}}

\newcommand{\good}[0]{\operatorname{good}}
\newcommand{\bad}[0]{\operatorname{bad}}

\swapnumbers \numberwithin{equation}{section}

\theoremstyle{plain}
\newtheorem{theorem}[equation]{Theorem}
\newtheorem{proposition}[equation]{Proposition}
\newtheorem{corollary}[equation]{Corollary}
\newtheorem{lemma}[equation]{Lemma}

\theoremstyle{definition}

\theoremstyle{remark}
\newtheorem{remark}[equation]{Remark}

\makeatletter
\@namedef{subjclassname@2010}{%
  \textup{2010} Mathematics Subject Classification}
\makeatother

\begin{document}

\title[Dyadic representation and $A_2$ theorem]{Representation of singular integrals by dyadic operators, and the $A_2$ theorem}

\author[T.~P.\ Hyt\"onen]{Tuomas P.\ Hyt\"onen}
\address{Department of Mathematics and Statistics, P.O.B.~68 (Gustaf H\"all\-str\"omin katu~2b), FI-00014 University of Helsinki, Finland}
\email{tuomas.hytonen@helsinki.fi}

%\date{\today}

%\keywords{}
% 42B20 Singular and oscillatory integrals (Calder\'on-Zygmund, etc.) 
% 42B25 Maximal functions, Littlewood-Paley theory
% 42B35 Function spaces arising in harmonic analysis 

% 46B09 Probabilistic methods in Banach space theory
% 46E40 Spaces of vector- and operator-valued functions
% 47A60 Functional calculus
% 47F05 Partial differential operators
% 60G46 Martingales and classical analysis

\maketitle

\begin{center}
{\small
Department of Mathematics and Statistics\\
P.O.B.~68 (Gustaf H\"all\-str\"omin katu~2b)\\
FI-00014 University of Helsinki, Finland\\
tuomas.hytonen@helsinki.fi
}
\end{center}

\begin{abstract}
This exposition presents a self-contained proof of the $A_2$ theorem, the quantitatively sharp norm inequality for singular integral operators in the weighted space $L^2(w)$. The strategy of the proof is a streamlined version of the author's original one, based on a probabilistic Dyadic Representation Theorem for singular integral operators. While more recent non-probabilistic approaches are also available now, the probabilistic method provides additional structural information, which has independent interest and other applications. The presentation emphasizes connections to the David--Journ\'e $T(1)$ theorem, whose proof is obtained as a byproduct. Only very basic Probability is used; in particular, the conditional probabilities of the original proof are completely avoided.\\

\end{abstract}

\newpage

\section{Introduction}

The goal of this exposition is to prove the following \emph{$A_2$ theorem}:

\begin{theorem}\label{thm:A2}
Let $T$ be any Calder\'on--Zygmund operator on $\R^d$ (like the Hilbert transform on $\R$, the Beurling transform on $\C\simeq\R^2$, or any of the Riesz transforms in $\R^d$ for $d\geq 2$; see Section~\ref{sec:representation} for the general definition).
Let $w:\R^d\to[0,\infty]$ be a weight in the Muckenhoupt class $A_2$, i.e.,
\begin{equation*}
   [w]_{A_2}:=\sup_Q\fint_Q w\cdot\fint_Q\frac{1}{w}<\infty\qquad\Big(\fint_Q w:=\frac{1}{\abs{Q}}\int_Q w\Big),
\end{equation*}
where the supremum is over all axes-parallel cubes $Q$ in $\R^d$.
Let $L^2(w)$ be the space of all measurable functions $f:\R^d\to\C$ such that
\begin{equation*}
  \Norm{f}{L^2(w)}:=\Big(\int_{\R^d}\abs{f}^2 w\Big)^{1/2}<\infty.
\end{equation*}
Then the following norm inequality is valid for any $f\in L^2(w)$, where $C_T$ only depends on $T$ and not on $f$ or $w$:
\begin{equation*}
  \Norm{Tf}{L^2(w)}\leq C_T\cdot[w]_{A_2}\cdot\Norm{f}{L^2(w)}.
\end{equation*}
\end{theorem}

This general theorem for all Calder\'on--Zygmund operators is due to the author~\cite{Hytonen:A2}, but it was first obtained in the listed special cases by S.~Petermichl and A.~Volberg \cite{PV} and Petermichl \cite{Petermichl:Hilbert,Petermichl:Riesz}, and in various further particular instances by a number of others \cite{CMP,Dragicevic:cubic,LPR,Vagharshakyan}. See also Section~\ref{sec:Beurling} for more details on the history of the problem.

Although several different proofs of Theorem~\ref{thm:A2} are known by now, I will present one that is a direct descendant of the original approach, but greatly streamlined in various places, based on ingredients from various subsequent proofs. On the large scale, I follow the strategy of my paper with C.~P\'erez, S.~Treil and A.~Volberg \cite{HPTV}, the first simplification of my original proof \cite{Hytonen:A2}. This consists of the following steps, which have independent interest:
\begin{enumerate}
  \item\label{it:red2Dyadic} Reduction to \emph{dyadic shift operators} (the Dyadic Representation Theorem): every Calder\'on--Zygmund operator $T$ has a (probabilistic) representation in terms of these simpler operators, and hence it suffices to prove a similar claim for every dyadic shift $S$ in place of $T$. This was a key novelty of \cite{Hytonen:A2} when it first appeared. In this exposition, the probabilistic ingredients of this representation have been simplified from \cite{Hytonen:A2,HPTV}, in that no conditional probabilities are needed.
  \item\label{it:red2Testing} Reduction to \emph{testing conditions} (a local $T(1)$ theorem): in order to have the full norm inequality
\begin{equation*}
    \Norm{Sf}{L^2(w)}\leq C_S[w]_{A_2}\Norm{f}{L^2(w)},
\end{equation*}
it suffices to have such an inequality for special test functions only:
\begin{equation*}
\begin{split}
    \Norm{S(1_Q w^{-1})}{L^2(w)} &\leq C_S[w]_{A_2}\Norm{1_Q w^{-1}}{L^2(w)}, \\
    \Norm{S^*(1_Q w)}{L^2(w^{-1})} &\leq C_S[w]_{A_2}\Norm{1_Q w}{L^2(w^{-1})}.
\end{split}
\end{equation*}
This goes essentially back to F.~Nazarov, Treil and Volberg \cite{NTV:2weightHaar}. (In the original proof \cite{Hytonen:A2}, in contrast to the simplification \cite{HPTV}, this reduction was done on the level of the Calder\'on--Zygmund operator, using a more difficult variant due to P\'erez, Treil and Volberg~\cite{PTV}).
   \item\label{it:verifyTesting} Verification of the testing conditions for $S$. This was first achieved by M.~T. Lacey, Petermichl and M.~C. Reguera~\cite{LPR}, although some adjustments were necessary to achieve the full generality in \cite{Hytonen:A2}.
\end{enumerate}

As said, several different proofs and extensions of the $A_2$ theorem have appeared over the past few years; see the final section for further discussion and references. In particular, it is now known that the probabilistic Dyadic Representation Theorem may be replaced by a deterministic Dyadic Domination Theorem. Its first version, a domination in norm, is due to A.~Lerner \cite{Lerner:domination}, and based on his clever local oscillation formula \cite{Lerner:formula}; this was subsequently improved to pointwise domination by J.~M. Conde-Alonso and G.~Rey \cite{CondeRey} and, independently, by Lerner and Nazarov \cite{LerNaz:book}. Yet another approach to the pointwise domination was found by Lacey \cite{Lacey:elem} and again simplified by Lerner \cite{Lerner:simplest}; this has the virtue of covering the biggest class of operators admissible for the $A_2$ theorem at the present state of knowledge. However, the probabilistic method continues to have its independent interest: it achieves the reduction to dyadic model operators as a linear \emph{identity}, in contrast to the (non-linear) \emph{upper bound} provided the deterministic domination. As such, it provides a structure theorem for singular integral operators, which has found other uses beyond the weighted norm inequalities, including the following:
\begin{itemize}
  \item The \emph{theorem itself} is applied to the estimation of \emph{commutators} of Calder\'on--Zygmund operators and BMO functions in a multi-parameter setting by L.~Dalenc and Y.~Ou~\cite{DalOu:iterated} and in a two-weight setting by I.~Holmes, M.~Lacey and B.~Wick \cite{HLW,HW}; it is also applied to sharp norm bounds for \emph{vector-valued extensions} of Calder\'on--Zygmund operators by S.~Pott and A.~Stoica~\cite{PS}.
  \item The \emph{methods behind this theorem} have been generalized by H.~Martikainen \cite{Martikainen:Advances} and Y.~Ou \cite{Ou:Tb} to the analysis of \emph{bi-parameter singular integrals}, yielding new $T(1)$ and $T(b)$ type theorems for these operators.
\end{itemize}

Whereas the domination method \emph{assumes} the unweighted $L^2$ boundedness of the operator $T$, the representation method can (and will, in this exposition) be set up in such a way that it \emph{derives} the unweighted boundedness from a priori weaker assumptions as a byproduct. Indeed, a proof of the $T(1)$ theorem of G.~David and J.-L. Journ\'e \cite{DJ} is obtained as a byproduct of the present exposition, and this approach was lifted to the nontrivial case of bi-parameter singular integrals in the mentioned works of Martikainen \cite{Martikainen:Advances} and Ou \cite{Ou:Tb}. Of course, the deterministic domination method has its own advantages, but the point that I want to make here is that so does the probabilistic approach, which I present in the following exposition.

\section{Preliminaries}

The standard (or reference) system of dyadic cubes is
\begin{equation*}
  \mathscr{D}^0:=\{2^{-k}([0,1)^d+m):k\in\Z,m\in\Z^d\}.
\end{equation*}
We will need several dyadic systems, obtained by translating the reference system as follows. Let $\omega=(\omega_j)_{j\in\Z}\in(\{0,1\}^d)^{\Z}$ and
\begin{equation*}
  I\dot+\omega:=I+\sum_{j:2^{-j}<\ell(I)}2^{-j}\omega_j.
\end{equation*}
Then
\begin{equation*}
  \mathscr{D}^{\omega}:=\{I\dot+\omega:I\in\mathscr{D}^0\},
\end{equation*}
and it is straightforward to check that $\mathscr{D}^{\omega}$ inherits the important nestedness property of $\mathscr{D}^0$: if $I,J\in\mathscr{D}^{\omega}$, then $I\cap J\in\{I,J,\varnothing\}$. When the particular $\omega$ is unimportant, the notation $\mathscr{D}$ is sometimes used for a generic dyadic system.

\subsection{Haar functions}
Any given dyadic system $\mathscr{D}$ has a natural function system associated to it: the Haar functions. In one dimension, there are two Haar functions associated with an interval $I$: the non-cancellative $h^0_I:=\abs{I}^{-1/2}1_I$ and the cancellative $h^1_I:=\abs{I}^{-1/2}(1_{I_{\ell}}-1_{I_r})$, where $I_{\ell}$ and $I_r$ are the left and right halves of $I$. In $d$ dimensions, the Haar functions on a cube $I=I_1\times\cdots\times I_d$ are formed of all the products of the one-dimensional Haar functions:
\begin{equation*}
  h_I^{\eta}(x)=h_{I_1\times\cdots\times I_d}^{(\eta_1,\ldots,\eta_d)}(x_1,\ldots,x_d):=\prod_{i=1}^d h_{I_i}^{\eta_i}(x_i).
\end{equation*}
The non-cancellative $h_I^0=\abs{I}^{-1/2}1_I$ has the same formula as in $d=1$. All other $2^d-1$ Haar functions $h_I^{\eta}$ with $\eta\in\{0,1\}^d\setminus\{0\}$ are cancellative, i.e., satisfy $\int h_I^{\eta}=0$, since they are cancellative in at least one coordinate direction.

For a fixed $\mathscr{D}$, all the cancellative Haar functions $h_I^{\eta}$, $I\in\mathscr{D}$ and $\eta\in\{0,1\}^d\setminus\{0\}$, form an orthonormal basis of $L^2(\R^d)$. Hence any function $f\in L^2(\R^d)$ has the orthogonal expansion
\begin{equation*}
  f=\sum_{I\in\mathscr{D}}\sum_{\eta\in\{0,1\}^d\setminus\{0\}}\pair{f}{h_I^{\eta}}h_I^{\eta}.
\end{equation*}
Since the different $\eta$'s seldom play any major role, this will be often abbreviated (with slight abuse of language) simply as
\begin{equation*}
  f=\sum_{I\in\mathscr{D}}\pair{f}{h_I}h_I,
\end{equation*}
and the summation over $\eta$ is understood implicitly.

\subsection{Dyadic shifts}

A dyadic shift with parameters $i,j\in\N:=\{0,1,2,\ldots\}$ is an operator of the form
\begin{equation*}
  Sf=\sum_{K\in\mathscr{D}}A_K f,\qquad
  A_K f=\sum_{\substack{I,J\in\mathscr{D};I,J\subseteq K \\ \ell(I)=2^{-i}\ell(K)\\ \ell(J)=2^{-j}\ell(K)}}a_{IJK}\pair{f}{h_I}h_J,
\end{equation*}
where $h_I$ is a Haar function on $I$ (similarly $h_J$), and the $a_{IJK}$ are coefficients with
\begin{equation*}
  \abs{a_{IJK}}\leq\frac{\sqrt{\abs{I}\abs{J}}}{\abs{K}}.
\end{equation*}
It is also required that all subshifts
\begin{equation*}
   S_{\mathscr{Q}}=\sum_{K\in\mathscr{Q}}A_K,\qquad\mathscr{Q}\subseteq\mathscr{D},
\end{equation*}
map $S_{\mathscr{Q}}:L^2(\R^d)\to L^2(\R^d)$ with norm at most one.

The shift is called cancellative, if all the $h_I$ and $h_J$ are cancellative; otherwise, it is called non-cancellative.

The notation $A_K$ indicates an ``averaging operator'' on $K$. Indeed, from the normalization of the Haar functions, it follows that
\begin{equation*}
  \abs{A_K f}\leq 1_K\fint_K\abs{f}
\end{equation*}
pointwise.

For cancellative shifts, the $L^2$ boundedness is automatic from the other conditions. This is a consequence of the following facts:
\begin{itemize}
  \item The pointwise bound for each $A_K$ implies that $\Norm{A_K f}{L^p}\leq\Norm{f}{L^p}$ for all $p\in[1,\infty]$; in particular, these components of $S$ are uniformly bounded on $L^2$ with norm one. (This first point is true even in the non-cancellative case.)
  \item Let $\D_K^{i}$ denote the orthogonal projection of $L^2$ onto $\lspan\{h_I:I\subseteq K,\ell(I)=2^{-i}\ell(K)\}$. When $i$ is fixed, it follows readily that any two $\D_K^{i}$ are orthogonal to each other. (This depends on the use of cancellative $h_I$.) Moreover, we have $A_K=\D_K^{j}A_K \D_K^{i}$. Then the boundedness of $S$ follows from two applications of Pythagoras' theorem with the uniform boundedness of the $A_K$ in between. 
\end{itemize}

A prime example of a non-cancellative shift (and the only one we need in these lectures) is the \emph{dyadic paraproduct}
\begin{equation*}
  \Pi_b f=\sum_{K\in\mathscr{D}}\pair{b}{h_K}\ave{f}_K h_K
  =\sum_{K\in\mathscr{D}}\abs{K}^{-1/2}\pair{b}{h_K}\cdot\pair{f}{h_K^0} h_K,
\end{equation*}
where $b\in\BMO_d$ (the dyadic BMO space) and $h_K$ is a cancellative Haar function. This is a dyadic shift with parameters $(i,j)=(0,0)$, where $a_{IJK}=\abs{K}^{-1/2}\pair{b}{h_K}$ for $I=J=K$. The $L^2$ boundedness of the paraproduct, if and only if $b\in\BMO_d$, is part of the classical theory. Actually, to ensure the normalization condition of the shift, it should be further required that $\Norm{b}{\BMO_d}\leq 1$.

\subsection{Random dyadic systems; good and bad cubes}

We obtain a notion of \emph{random dyadic systems} by equipping the parameter set $\Omega:=(\{0,1\}^d)^{\Z}$ with the natural probability measure: each component $\omega_j$ has an equal probability $2^{-d}$ of taking any of the $2^d$ values in $\{0,1\}^d$, and all components are independent of each other.

Let $\phi:[0,1]\to[0,1]$ be a fixed \emph{modulus of continuity}: a strictly increasing function with $\phi(0)=0$, $\phi(1)=1$, and $t\mapsto\phi(t)/t$ decreasing (hence $\phi(t)\geq t$) with $\lim_{t\to 0}\phi(t)/t=\infty$.
We further require the \emph{Dini condition}
\begin{equation*}
  \int_0^1\phi(t)\frac{\ud t}{t}<\infty.
\end{equation*}
Main examples include $\phi(t)=t^{\gamma}$ with $\gamma\in(0,1)$ and
\begin{equation*}
  \phi(t)=\Big(1+\frac{1}{\gamma}\log\frac{1}{t}\Big)^{-\gamma},\qquad \gamma>1.
\end{equation*}
We also fix a (large) parameter $r\in\Z_+$. (How large, will be specified shortly.)

A cube $I\in\mathscr{D}^{\omega}$ is called bad if there exists $J\in\mathscr{D}^{\omega}$ such that $\ell(J)\geq 2^r\ell(I)$ and
\begin{equation*}
  \dist(I,\partial J)\leq\phi\Big(\frac{\ell(I)}{\ell(J)}\Big)\ell(J):
\end{equation*}
roughly, $I$ is relatively close to the boundary of a much bigger cube.

\begin{remark}
This definition of good cubes goes back to Nazarov--Treil--Volberg \cite{NTV:Tb} in the context of singular integrals with respect to non-doubling measures. They used the modulus of continuity $\phi(t)=t^{\gamma}$, where $\gamma$ was chosen to depend on the dimension and the H\"older exponent of the Calder\'on--Zygmund kernel via
\begin{equation*}
  \gamma=\frac{\alpha}{2(d+\alpha)}.
\end{equation*}
This choice has become ``canonical'' in the subsequent literature, including the original proof of the $A_2$ theorem. However, other choices can also be made, as we do here.
\end{remark}

We make some basic probabilistic observations related to badness. Let $I\in\mathscr{D}^0$ be a reference interval. The \emph{position} of the translated interval
\begin{equation*}
  I\dot+\omega=I+\sum_{j:2^{-j}<\ell(I)}2^{-j}\omega_j,
\end{equation*}
by definition, depends only on $\omega_j$ for $2^{-j}<\ell(I)$. On the other hand, the \emph{badness} of $I\dot+\omega$ depends on its \emph{relative position} with respect to the bigger intervals
\begin{equation*}
  J\dot+\omega=J+\sum_{j:2^{-j}<\ell(I)}2^{-j}\omega_j+\sum_{j:\ell(I)\leq 2^{-j}<\ell(I)}2^{-j}\omega_j.
\end{equation*}
The same translation component $\sum_{j:2^{-j}<\ell(I)}2^{-j}\omega_j$ appears in both $I\dot+\omega$ and $J\dot+\omega$, and so does not affect the relative position of these intervals. Thus this relative position, and hence the badness of $I$, depends only on $\omega_j$ for $2^{-j}\geq\ell(I)$. In particular:

\begin{lemma}\label{lem:indep}
For $I\in\mathscr{D}^0$, the position and badness of $I\dot+\omega$ are independent random variables.
\end{lemma}

Another observation is the following: by symmetry and the fact that the condition of badness only involves relative position and size of different cubes, it readily follows that the probability of a particular cube $I\dot+\omega$ being bad is equal for all cubes $I\in\mathscr{D}^0$:
\begin{equation*}
  \prob_{\omega}(I\dot+\omega\bad)=\pi_{\bad}=\pi_{\bad}(r,d,\phi).
\end{equation*}

The final observation concerns the value of this probability:

\begin{lemma}
We have
\begin{equation*}
  \pi_{\bad}\leq 8d\int_0^{2^{-r}}\phi(t)\frac{\ud t}{t};
\end{equation*}
in particular, $\pi_{\bad}<1$ if $r=r(d,\phi)$ is chosen large enough.
\end{lemma}

With $r=r(d,\phi)$ chosen like this, we then have $\pi_{\good}:=1-\pi_{\bad}>0$, namely, good situations have positive probability!

\begin{proof}
Observe that in the definition of badness, we only need to consider those $J$ with $I\subseteq J$. Namely, if $I$ is close to the boundary of some bigger $J$, we can always find another dyadic $J'$ of the same size as $J$ which contains $I$, and then $I$ will also be close to the boundary of $J'$. Hence we need to consider the relative position of $I$ with respect to each $J\supset I$ with $\ell(J)=2^k\ell(I)$ and $k=r,r+1,\ldots$ For a fixed $k$, this relative position is determined by
\begin{equation*}
  \sum_{j:\ell(I)\leq 2^{-j}<2^k\ell(I)}2^{-j}\omega_j,
\end{equation*}
which has $2^{kd}$ different values with equal probability. These correspond to the subcubes of $J$ of size $\ell(I)$.

Now bad position of $I$ are those which are within distance $\phi(\ell(I)/\ell(J))\cdot\ell(J)$ from the boundary. Since the possible position of the subcubes are discrete, being integer multiples of $\ell(I)$, the effective bad boundary region has depth
\begin{equation*}
\begin{split}
  \Big\lceil \phi\Big(\frac{\ell(I)}{\ell(J)}\Big)\frac{\ell(J)}{\ell(I)}\Big\rceil\ell(I)
  &\leq\Big(\phi\Big(\frac{\ell(I)}{\ell(J)}\Big)\frac{\ell(J)}{\ell(I)}+1\Big)\ell(I) \\
  &=\ell(J)\Big(\phi\Big(\frac{\ell(I)}{\ell(J)}\Big)+\frac{\ell(I)}{\ell(J)}\Big)\leq 2\ell(J)\phi\Big(\frac{\ell(I)}{\ell(J)}\Big),
\end{split}
\end{equation*}
by using that $t\leq\phi(t)$.

The good region is the cube inside $J$, whose side-length is $\ell(J)$ minus twice the depth of the bad boundary region:
\begin{equation*}
  \ell(J)-2\Big\lceil \phi\Big(\frac{\ell(I)}{\ell(J)}\Big)\frac{\ell(J)}{\ell(I)}\Big\rceil\ell(I)
  \geq\ell(J)-4\ell(J)\phi\Big(\frac{\ell(I)}{\ell(J)}\Big).
\end{equation*}
Hence the volume of the bad region is
\begin{equation*}
\begin{split}
  \abs{J}-\Big(\ell(J)-2\Big\lceil \phi\Big(\frac{\ell(I)}{\ell(J)}\Big)\frac{\ell(J)}{\ell(I)}\Big\rceil\ell(I)\Big)^d
  &\leq\abs{J}\Big(1-\Big(1-4\phi\Big(\frac{\ell(I)}{\ell(J)}\Big)\Big)^d\Big) \\
  &\leq\abs{J}\cdot 4d\phi\Big(\frac{\ell(I)}{\ell(J)}\Big)
\end{split}
\end{equation*}
by the elementary inequality $(1-\alpha)^d\geq 1-\alpha d$ for $\alpha\in[0,1]$. (We assume that $r$ is at least so large that $4\phi(2^{-r})\leq 1$.)

So the fraction of the bad region of the total volume is at most $4d\phi(\ell(I)/\ell(J))=4d\phi(2^{-k})$ for a fixed $k=r,r+1,\ldots$. This gives the final estimate
\begin{equation*}
\begin{split}
  \prob_{\omega}(I\dot+\omega\bad)
  &\leq\sum_{k=r}^{\infty}4d\phi(2^{-k})
  =\sum_{k=r}^{\infty}8d\frac{\phi(2^{-k})}{2^{-k}}2^{-k-1} \\
  &\leq\sum_{k=r}^{\infty}8d\int_{2^{-k-1}}^{2^{-k}}\frac{\phi(t)}{t}\ud t
   =8d\int_0^{2^{-r}}\phi(t)\frac{\ud t}{t},
\end{split}
\end{equation*}
where we used that $\phi(t)/t$ is decreasing in the last inequality.
\end{proof}

\section{The dyadic representation theorem}\label{sec:representation}

Let $T$ be a Calder\'on--Zygmund operator on $\R^d$. That is, it acts on a suitable dense subspace of functions in $L^2(\R^d)$ (for the present purposes, this class should at least contain the indicators of cubes in $\R^d$) and has the kernel representation
\begin{equation*}
  Tf(x)=\int_{\R^d}K(x,y)f(y)\ud y,\qquad x\notin\supp f.
\end{equation*}
Moreover, the kernel should satisfy the \emph{standard estimates}, which we here assume in a slightly more general form than usual, involving another modulus of continuity~$\psi$, like the one considered above:
\begin{equation*}
\begin{split}
  \abs{K(x,y)} &\leq\frac{C_0}{\abs{x-y}^d}, \\
  \abs{K(x,y)-K(x',y)}+\abs{K(y,x)-K(y,x')}
  &\leq\frac{C_\psi}{\abs{x-y}^d}\psi\Big(\frac{\abs{x-x'}}{\abs{x-y}}\Big)
\end{split}
\end{equation*}
for all $x,x',y\in\R^d$ with $\abs{x-y}>2\abs{x-x'}$. Let us denote the smallest admissible constants $C_0$ and $C_\psi$ by $\Norm{K}{CZ_0}$ and $\Norm{K}{CZ_\psi}$. The classical standard estimates correspond to the choice $\psi(t)=t^{\alpha}$, $\alpha\in(0,1]$, in which case we write $\Norm{K}{CZ_\alpha}$ for $\Norm{K}{CZ_\psi}$.

We say that $T$ is a bounded Calder\'on--Zygmund operator, if in addition $T:L^2(\R^d)\to L^2(\R^d)$, and we denote its operator norm by $\Norm{T}{L^2\to L^2}$.

Let us agree that $\abs{\ }$ stands for the $\ell^{\infty}$ norm on $\R^d$, i.e., $\abs{x}:=\max_{1\leq i\leq d}\abs{x_i}$. While the choice of the norm is not particularly important, this choice is slightly more convenient than the usual Euclidean norm when dealing with cubes as we will: e.g., the diameter of a cube in the $\ell^{\infty}$ norm is equal to its sidelength $\ell(Q)$.

Let us first formulate the dyadic representation theorem for general moduli of continuity, and then specialize it to the usual standard estimates. Define the following coefficients for $i,j\in\N$:
\begin{equation*}
   \tau(i,j):=\phi(2^{-\max\{i,j\}})^{-d}\psi\big(2^{-\max\{i,j\}}\phi(2^{-\max\{i,j\}})^{-1}\big),
\end{equation*}
if $\min\{i,j\}>0$; and
\begin{equation*}
  \tau(i,j):= \Psi\big(2^{-\max\{i,j\}}\phi(2^{-\max\{i,j\}})^{-1}\big),\qquad
   \Psi(t):=\int_0^t\psi(s)\frac{\ud s}{s},
\end{equation*}
if $\min\{i,j\}=0$.

We assume that $\phi$ and $\psi$ are such, that
\begin{equation}\label{eq:decay}
  \sum_{i,j=0}^{\infty}\tau(i,j)\eqsim\int_0^1\frac{1}{\phi(t)^{d}}\psi\Big(\frac{t}{\phi(t)}\Big)\frac{\ud t}{t}+\int_0^1\Psi\Big(\frac{t}{\phi(t)}\Big)\frac{\ud t}{t}<\infty.
\end{equation}
This is the case, in particular, when $\psi(t)=t^{\alpha}$ (usual standard estimates) and $\phi(t)=(1+a^{-1}\log t^{-1})^{-\gamma}$; then one checks that
\begin{equation*}
  \tau(i,j)\lesssim P(\max\{i,j\})2^{-\alpha\max\{i,j\}},\qquad P(j)=(1+j)^{\gamma(d+\alpha)},
\end{equation*}
which clearly satisfies the required convergence. However, it is also possible to treat weaker forms of the standard estimates with a logarithmic modulus $\psi(t)=(1+a^{-1}\log t^{-1})^{-\alpha}$. This might be of some interest for applications, but we do not pursue this line any further here.

\begin{theorem}\label{thm:formula}
Let $T$ be a bounded Calder\'on--Zygmund operator with modulus of continuity satisfying the above assumption. Then it has an expansion, say for $f,g\in C^1_c(\R^d)$,
\begin{equation*}
  \pair{g}{Tf}
  =c\cdot\big(\Norm{T}{L^2\to L^2}+\Norm{K}{CZ_\psi}\big)\cdot\Exp_{\omega} \sum_{i,j=0}^{\infty} \tau(i,j)\pair{g}{S^{ij}_{\omega}f},
\end{equation*}
where $c$ is a dimensional constant and $S^{ij}_{\omega}$ is a dyadic shift of parameters $(i,j)$ on the dyadic system $\mathscr{D}^{\omega}$; all of them except possibly $S^{00}_{\omega}$ are cancellative.
\end{theorem}

The first version of this theorem appeared in \cite{Hytonen:A2}, and another one in \cite{HPTV}.
The present proof is yet another variant of the same argument. It is slightly simpler in terms of the probabilistic tools that are used: no conditional probabilities are needed, although they were important for the original arguments.

In proving this theorem, we do not actually need to employ the full strength of the assumption that $T:L^2(\R^d)\to L^2(\R^d)$; rather it suffices to have the kernel conditions plus the following conditions of the $T1$ theorem of David--Journ\'e:
\begin{equation*}
\begin{split}
  \abs{\pair{1_Q}{T1_Q}} &\leq C_{WBP}\abs{Q}\quad\text{(weak boundedness property)},\\
   & T1\in\BMO(\R^d),\quad T^*1\in\BMO(\R^d).
\end{split}
\end{equation*}
Let us denote the smallest $C_{WBP}$ by $\Norm{T}{WBP}$. Then we have the following more precise version of the representation:

\begin{theorem}\label{thm:formula2}
Let $T$ be a Calder\'on--Zygmund operator with modulus of continuity satisfying the above assumption. Then it has an expansion, say for $f,g\in C^1_c(\R^d)$,
\begin{equation*}
\begin{split}
  \pair{g}{Tf}
  &=c\cdot\big(\Norm{K}{CZ_0}+\Norm{K}{CZ_\psi}\big)\Exp_{\omega} \sum_{\substack{i,j=0\\ \max\{i,j\}> 0}}^{\infty} \tau(i,j)\pair{g}{S^{ij}_{\omega}f} \\
  &+c\cdot\big(\Norm{K}{CZ_0}+\Norm{T}{WBP}\big)\Exp_{\omega}\pair{g}{S^{00}_{\omega}f}
    +\Exp_{\omega}\pair{g}{\Pi_{T1}^{\omega}f}+\Exp_{\omega}\pair{g}{(\Pi_{T^*1}^{\omega})^*f}
\end{split}
\end{equation*}
where $S^{ij}_{\omega}$ is a cancellative dyadic shift of parameters $(i,j)$ on the dyadic system $\mathscr{D}^{\omega}$, and $\Pi_{b}^{\omega}$ is a dyadic paraproduct on the dyadic system $\mathscr{D}^{\omega}$ associated with the $\BMO$-function $b\in\{T1,T^*1\}$.
\end{theorem}

\begin{remark}
Note that $\Pi^{\omega}_b=\Norm{b}{\BMO}\cdot S^{\omega}_b$, where $S^{\omega}_b=\Pi^{\omega}_b/\Norm{b}{\BMO}$ is a shift with the correct normalization. Hence, writing everything in terms of normalized shifts, as in Theorem~\ref{thm:formula}, we get the factor $\Norm{T1}{\BMO}\lesssim\Norm{T}{L^2\to L^2}+\Norm{K}{CZ_\psi}$ in the second-to-last term, and $\Norm{T^*1}{\BMO}\lesssim\Norm{T}{L^2\to L^2}+\Norm{K}{CZ_\psi}$ in the last one.
The proof will also show that both occurrences of the factor $\Norm{K}{CZ_0}$ could be replaced by $\Norm{T}{L^2\to L^2}$, giving the statement of Theorem~\ref{thm:formula} (since trivially $\Norm{T}{WBP}\leq\Norm{T}{L^2\to L^2}$).
\end{remark}

As a by-product, Theorem~\ref{thm:formula2} delivers a proof of the $T1$ theorem: under the above assumptions, the operator $T$ is already bounded on $L^2(\R^d)$. Namely, all the dyadic shifts $S^{ij}_{\omega}$ are uniformly bounded on $L^2(\R^d)$ by definition, and the convergence condition \eqref{eq:decay} ensures that so is their average representing the operator $T$. This by-product proof of the $T1$ theorem is not a coincidence, since the proof of Theorems~\ref{thm:formula} and \ref{thm:formula2} was actually inspired by the proof of the $T1$ theorem for non-doubling measures due to Nazarov--Treil--Volberg \cite{NTV:Tb} and its vector-valued extension \cite{Hytonen:nonhomog}.

A key to the proof of the dyadic representation is a random expansion of $T$ in terms of Haar functions $h_I$, where the bad cubes are avoided:

\begin{proposition}
\begin{equation*}
  \pair{g}{Tf}
  =\frac{1}{\pi_{\good}}\Exp_{\omega}\sum_{I,J\in\mathscr{D}^{\omega}}1_{\good}(\operatorname{smaller}\{I,J\})\cdot
  \pair{g}{h_{J}}\pair{h_{J}}{Th_{I}}\pair{h_{I}}{f},
\end{equation*}
where
\begin{equation*}
  \operatorname{smaller}\{I,J\}:=\begin{cases} I & \text{if }\ell(I)\leq\ell(J), \\ J & \text{if }\ell(J)>\ell(I). \end{cases}
\end{equation*}
\end{proposition}

\begin{proof}
Recall that
\begin{equation*}
  f=\sum_{I\in\mathscr{D}^0}\pair{f}{h_{I\dot+\omega}}h_{I\dot+\omega}
\end{equation*}
for any fixed $\omega\in\Omega$; and we can also take the expectation $\Exp_{\omega}$ of both sides of this identity.

Let
\begin{equation*}
  1_{\good}(I\dot+\omega):=\begin{cases} 1, & \text{if $I\dot+\omega$ is good},\\ 0, & \text{else}\end{cases}
\end{equation*}
We make use of the above random Haar expansion of $f$, multiply and divide by
\begin{equation*}
  \pi_{\good}=\prob_{\omega}(I\dot+\omega\good)=\Exp_{\omega}1_{\good}(I\dot+\omega),
\end{equation*}
and use the independence from Lemma~\ref{lem:indep} to get:
\begin{equation*}
\begin{split}
  \pair{g}{Tf}
  &=\Exp_{\omega}\sum_{I}\pair{g}{Th_{I\dot+\omega}}\pair{h_{I\dot+\omega}}{f} \\
  &=\frac{1}{\pi_{\good}}\sum_{I}\Exp_{\omega}[1_{\good}(I\dot+\omega)] \Exp_{\omega}[\pair{g}{Th_{I\dot+\omega}}\pair{h_{I\dot+\omega}}{f}] \\
  &=\frac{1}{\pi_{\good}}\Exp_{\omega}\sum_{I}1_{\good}(I\dot+\omega) \pair{g}{Th_{I\dot+\omega}}\pair{h_{I\dot+\omega}}{f}  \\
  &=\frac{1}{\pi_{\good}}\Exp_{\omega}\sum_{I,J}1_{\good}(I\dot+\omega) \pair{g}{h_{J\dot+\omega}}\pair{h_{J\dot+\omega}}{Th_{I\dot+\omega}}\pair{h_{I\dot+\omega}}{f}.
\end{split}
\end{equation*}
On the other hand, using independence again in half of this double sum, we have
\begin{equation*}
\begin{split}
  &\frac{1}{\pi_{\good}}\sum_{\ell(I)>\ell(J)}\Exp_{\omega}[1_{\good}(I\dot+\omega) \pair{g}{h_{J\dot+\omega}}\pair{h_{J\dot+\omega}}{Th_{I\dot+\omega}}\pair{h_{I\dot+\omega}}{f} ] \\
  &=\frac{1}{\pi_{\good}}\sum_{\ell(I)>\ell(J)}\Exp_{\omega}[1_{\good}(I\dot+\omega)]
     \Exp_{\omega}[ \pair{g}{h_{J\dot+\omega}}\pair{h_{J\dot+\omega}}{Th_{I\dot+\omega}}\pair{h_{I\dot+\omega}}{f} ] \\
   &= \Exp_{\omega}\sum_{\ell(I)>\ell(J)}
    \pair{g}{h_{J\dot+\omega}}\pair{h_{J\dot+\omega}}{Th_{I\dot+\omega}}\pair{h_{I\dot+\omega}}{f},
\end{split}
\end{equation*}
and hence
\begin{equation*}
\begin{split}
    \pair{g}{Tf}
    &= \frac{1}{\pi_{\good}}\Exp_{\omega}\sum_{\ell(I)\leq\ell(J)}
        1_{\good}(I\dot+\omega) \pair{g}{h_{J\dot+\omega}}\pair{h_{J\dot+\omega}}{Th_{I\dot+\omega}}\pair{h_{I\dot+\omega}}{f} \\
     &\qquad+\Exp_{\omega}\sum_{\ell(I)>\ell(J)}
        \pair{g}{h_{J\dot+\omega}}\pair{h_{J\dot+\omega}}{Th_{I\dot+\omega}}\pair{h_{I\dot+\omega}}{f}.
\end{split}
\end{equation*}
Comparison with the basic identity
\begin{equation}\label{eq:basic}
  \pair{g}{Tf}
  =\Exp_{\omega}\sum_{I,J}\pair{g}{h_{J\dot+\omega}}\pair{h_{J\dot+\omega}}{Th_{I\dot+\omega}}\pair{h_{I\dot+\omega}}{f}
\end{equation}
shows that
\begin{equation*}
\begin{split}
  &\Exp_{\omega}\sum_{\ell(I)\leq\ell(J)}
        \pair{g}{h_{J\dot+\omega}}\pair{h_{J\dot+\omega}}{Th_{I\dot+\omega}}\pair{h_{I\dot+\omega}}{f} \\
    &= \frac{1}{\pi_{\good}}\Exp_{\omega}\sum_{\ell(I)\leq\ell(J)}
        1_{\good}(I\dot+\omega) \pair{g}{h_{J\dot+\omega}}\pair{h_{J\dot+\omega}}{Th_{I\dot+\omega}}\pair{h_{I\dot+\omega}}{f}.
\end{split}  
\end{equation*}
Symmetrically, we also have
\begin{equation*}
\begin{split}
  &\Exp_{\omega}\sum_{\ell(I)>\ell(J)}
        \pair{g}{h_{J\dot+\omega}}\pair{h_{J\dot+\omega}}{Th_{I\dot+\omega}}\pair{h_{I\dot+\omega}}{f} \\
    &= \frac{1}{\pi_{\good}}\Exp_{\omega}\sum_{\ell(I)>\ell(J)}
        1_{\good}(J\dot+\omega) \pair{g}{h_{J\dot+\omega}}\pair{h_{J\dot+\omega}}{Th_{I\dot+\omega}}\pair{h_{I\dot+\omega}}{f},
\end{split}  
\end{equation*}
and this completes the proof.
\end{proof}

This is essentially the end of probability in this proof. Henceforth, we can simply concentrate on the summation inside $\Exp_{\omega}$, for a fixed value of $\omega\in\Omega$, and manipulate it into the required form. Moreover, we will concentrate on the half of the sum with $\ell(J)\geq\ell(I)$, the other half being handled symmetrically. We further divide this sum into the following parts:
\begin{equation*}
\begin{split}
  \sum_{\ell(I)\leq\ell(J)}
  &=\sum_{\dist(I,J)>\ell(J)\phi(\ell(I)/\ell(J))}
    +\sum_{I\subsetneq J}+\sum_{I=J}
    +\sum_{\substack{\dist(I,J)\leq\ell(J)\phi(\ell(I)/\ell(J))\\ I\cap J=\varnothing}} \\
  &=:\sigma_{\operatorname{out}}+\sigma_{\operatorname{in}}+\sigma_{=}+\sigma_{\operatorname{near}}.
\end{split}
\end{equation*}
In order to recognize these series as sums of dyadic shifts, we need to locate, for each pair $(I,J)$ appearing here, a common dyadic ancestor which contains both of them. The existence of such containing cubes, with control on their size, is provided by the following:

\begin{lemma}\label{lem:IveeJ}
If $I\in\mathscr{D}$ is good and $J\in\mathscr{D}$ is a disjoint ($J\cap I=\varnothing$) cube with $\ell(J)\geq\ell(I)$, then there exists $K\supseteq I\cup J$ which satisfies
\begin{equation*}
\begin{split}
  \ell(K) &\leq 2^r\ell(I), \qquad\text{if}\qquad \dist(I,J)\leq\ell(J)\phi\Big(\frac{\ell(I)}{\ell(J)}\Big), \\
  \ell(K)\phi\Big(\frac{\ell(I)}{\ell(K)}\Big) &\leq 2^r\dist(I,J), \qquad\text{if}\qquad\dist(I,J)>\ell(J)\phi\Big(\frac{\ell(I)}{\ell(J)}\Big).
\end{split}
\end{equation*}
\end{lemma}

\begin{proof}
Let us start with the following initial observation: if $K\in\mathscr{D}$ satisfies $I\subseteq K$, $J\subset K^c$, and $\ell(K)\geq 2^r\ell(I)$, then
\begin{equation*}
  \ell(K)\phi\Big(\frac{\ell(I)}{\ell(K)}\Big)<\dist(I,\partial K)=\dist(I,K^c)\leq\dist(I,J).
\end{equation*}

\subsubsection*{Case $\dist(I,J)\leq\ell(J)\phi(\ell(I)/\ell(J))$}
As $I\cap J=\varnothing$, we have $\dist(I,J)=\dist(I,\partial J)$, and since $I$ is good, this implies $\ell(J)< 2^r\ell(I)$.
Let $K=I^{(r)}$, and assume for contradiction that $J\subset K^c$. Then the initial observation implies that
\begin{equation*}
  \ell(K)\phi\Big(\frac{\ell(I)}{\ell(K)}\Big)<\dist(I,J)\leq\ell(J)\phi\Big(\frac{\ell(I)}{\ell(J)}\Big).
\end{equation*}
Dividing both sides by $\ell(I)$ and recalling that $\phi(t)/t$ is decreasing, this implies that $\ell(K)<\ell(J)$, a contradiction with $\ell(K)=2^r\ell(I)>\ell(J)$.
Hence $J\not\subset K^c$, and since $\ell(J)<\ell(K)$, this implies that $J\subset K$.

\subsubsection*{Case $\dist(I,J)>\ell(J)\phi(\ell(I)/\ell(J))$}
Consider the minimal $K\supset I$ with $\ell(K)\geq 2^r\ell(I)$ and $\dist(I,J)\leq\ell(K)\phi(\ell(I)/\ell(K))$. (Since $\phi(t)/t\to\infty$ as $t\to 0$, this bound holds for all large enough $K$.) Then (since $\phi(t)/t$ is decreasing) $\ell(K)>\ell(J)$, and by the initial observation, $J\not\subset K^c$. Hence either $J\subset K$, and it suffices to estimate $\ell(K)$.

By the minimality of $K$, there holds at least one of
\begin{equation*}
  \tfrac12\ell(K)<2^r\ell(I)\qquad\text{or}\qquad \tfrac12\ell(K)\phi\Big(\frac{\ell(I)}{\tfrac12\ell(K)}\Big)<\dist(I,J),
\end{equation*}
and the latter immediately implies that $\ell(K)\phi(\ell(I)/\ell(K))<2\dist(I,J)$. In the first case, since $\ell(I)\leq\ell(J)\leq\ell(K)$, we have
\begin{equation*}
  \ell(K)\phi\Big(\frac{\ell(I)}{\ell(K)}\Big)
  \leq 2^r\ell(I)\Big(\frac{\ell(I)}{\ell(K)}\Big)
  \leq 2^r\ell(J)\Big(\frac{\ell(I)}{\ell(J)}\Big)
  <2^r\dist(I,J),
\end{equation*}
so the required bound is true in each case.
\end{proof}

We denote the minimal such $K$ by $I\vee J$, thus
\begin{equation*}
  I\vee J:=\bigcap_{K\supseteq I\cup J} K.
\end{equation*}

\subsection{Separated cubes, $\sigma_{\operatorname{out}}$}

We reorganize the sum $\sigma_{\operatorname{out}}$ with respect to the new summation variable $K=I\vee J$, as well as the relative size of $I$ and $J$ with respect to $K$:
\begin{equation*}
  \sigma_{\operatorname{out}}
  =\sum_{j=1}^{\infty}\sum_{i=j}^{\infty}\sum_K \sum_{\substack{\dist(I,J)>\ell(J)\phi(\ell(I)/\ell(J))\\ I\vee J=K \\ \ell(I)=2^{-i}\ell(K), \ell(J)=2^{-j}\ell(K)}}.
\end{equation*}
Note that we can start the summation from $1$ instead of $0$, since the disjointness of $I$ and $J$ implies that $K=I\vee J$ must be strictly larger than either of $I$ and $J$.
The goal is to identify the quantity in parentheses as a decaying factor times a cancellative averaging operator with parameters $(i,j)$.

\begin{lemma}
For $I$ and $J$ appearing in $\sigma_{\operatorname{out}}$, we have
\begin{equation*}
  \abs{\pair{h_J}{Th_I}}
  \lesssim\Norm{K}{CZ_\psi}\frac{\sqrt{\abs{I}\abs{J}}}{\abs{K}}\phi\Big(\frac{\ell(I)}{\ell(K)}\Big)^{-d}\psi\Big(\frac{\ell(I)}{\ell(K)}\phi\Big(\frac{\ell(I)}{\ell(K)}\Big)^{-1}\Big),
  \quad K=I\vee J.
\end{equation*}
\end{lemma}

\begin{proof}
Using the cancellation of $h_I$, standard estimates, and Lemma~\ref{lem:IveeJ}
\begin{equation*}
\begin{split}
  \abs{\pair{h_J}{Th_I}}
  &=\Babs{\iint h_J(x)K(x,y)h_I(y)\ud y\ud x} \\
  &=\Babs{\iint h_J(x)[K(x,y)-K(x,y_I)]h_I(y)\ud y\ud x} \\
  &\lesssim\Norm{K}{CZ_\psi}\iint \abs{h_J(x)}\frac{1}{\dist(I,J)^d}\psi\Big(\frac{\ell(I)}{\dist(I,J)}\Big)\abs{h_I(y)}\ud y\ud x \\
  &=\Norm{K}{CZ_\psi}\frac{1}{\dist(I,J)^d}\psi\Big(\frac{\ell(I)}{\dist(I,J)}\Big)\Norm{h_J}{1}\Norm{h_I}{1} \\
  &\lesssim\Norm{K}{CZ_\psi}\frac{1}{\ell(K)^d}\phi\Big(\frac{\ell(I)}{\ell(K)}\Big)^{-d}\psi\Big(\frac{\ell(I)}{\ell(K)}\phi\Big(\frac{\ell(I)}{\ell(K)}\Big)^{-1}\Big)\sqrt{\abs{J}}\sqrt{\abs{I}}.\qedhere
\end{split}
\end{equation*}
\end{proof}

\begin{lemma}
\begin{equation*}
\begin{split}
  \sum_{\substack{\dist(I,J)>\ell(J)\phi(\ell(I)/\ell(J))\\ I\vee J=K \\ \ell(I)=2^{-i}\ell(K)\leq \ell(J)=2^{-j}\ell(K)}}
  &1_{\good}(I)\cdot  \pair{g}{h_{J}}\pair{h_{J}}{Th_{I}}\pair{h_{I}}{f}  \\
  &=\Norm{K}{CZ_\psi}\phi(2^{-i})^{-d}\psi\big(2^{-i}\phi(2^{-i})^{-1}\big)\pair{g}{A_K^{ij}f},
\end{split}
\end{equation*}
where $A_K^{ij}$ is a cancellative averaging operator with parameters $(i,j)$.
\end{lemma}

\begin{proof}
By the previous lemma, substituting $\ell(I)/\ell(K)=2^{-i}$,
\begin{equation*}
    \abs{\pair{h_J}{Th_I}}
  \lesssim\Norm{K}{CZ_\psi}\frac{\sqrt{\abs{I}\abs{J}}}{\abs{K}}\phi(2^{-i})^{-d}\psi\big(2^{-i}\phi(2^{-i})^{-1}\big),
\end{equation*}
and the first factor is precisely the required size of the coefficients of $A_K^{ij}$.
\end{proof}

Summarizing, we have
\begin{equation*}
  \sigma_{\operatorname{out}}
  =\Norm{K}{CZ_\psi}\sum_{j=1}^{\infty}\sum_{i=j}^{\infty}\phi(2^{-i})^{-d}\psi\big(2^{-i}\phi(2^{-i})^{-1}\big)\pair{g}{S^{ij}f}.
\end{equation*}

\subsection{Contained cubes, $\sigma_{\operatorname{in}}$}

When $I\subsetneq J$, then $I$ is contained in some subcube of $J$, which we denote by $J_I$.

\begin{equation*}
\begin{split}
  \pair{h_J}{Th_I}
  &=\pair{1_{J_I^c}h_J}{Th_I}+\pair{1_{J_I}h_J}{Th_I} \\
  &=\pair{1_{J_I^c}h_J}{Th_I}+\ave{h_J}_{J_I}\pair{1_{J_I}}{Th_I} \\
  &=\pair{1_{J_I^c}(h_J-\ave{h_J}_{J_I})}{Th_I}+\ave{h_J}_{I}\pair{1}{Th_I},
\end{split}
\end{equation*}
where we noticed that $h_J$ is constant on $J_I\supseteq I$.

\begin{lemma}
\begin{equation*}
  \abs{ \pair{1_{J_I^c}(h_J-\ave{h_J}_{J_I})}{Th_I} }
  \lesssim\big(\Norm{K}{CZ_0}+\Norm{K}{CZ_\psi}\big)\Big(\frac{\abs{I}}{\abs{J}}\Big)^{1/2}\Psi\Big(\frac{\ell(I)}{\ell(J)}\phi\big(\frac{\ell(I)}{\ell(J)}\big)^{-1}\Big),
\end{equation*}
where
\begin{equation*}
  \Psi(r):=\int_0^r\psi(t)\frac{\ud t}{t},
\end{equation*}
and $\Norm{K}{CZ_0}$ could be alternatively replaced by $\Norm{T}{L^2\to L^2}$.
\end{lemma}

\begin{proof}
\begin{equation*}
  \abs{ \pair{1_{J_I^c}(h_J-\ave{h_J}_{J_I})}{Th_I} }
  \leq 2\Norm{h_J}{\infty}\int_{J_I^c}\abs{Th_I(x)}\ud x,
\end{equation*}
where $\Norm{h_J}{\infty}=\abs{J}^{-1/2}$.

\subsubsection*{Case $\ell(I)\geq 2^{-r}\ell(J)$}
We have
\begin{equation*}
\begin{split}
  \int_{J_I^c}\abs{Th_I(x)}\ud x
  &\leq\int_{3I\setminus I}\Babs{\int K(x,y)h_I(y)\ud y}\ud x \\
  &\qquad+\int_{(3I)^c}\Babs{\int [K(x,y)-K(x,y_I)]h_I(y)\ud y}\ud x \\
  &\lesssim\Norm{K}{CZ_0}\int_{3I\setminus I}\int_I\frac{1}{\abs{x-y}^d}\ud y\ud x \Norm{h_I}{\infty} \\
  &\qquad+\Norm{K}{CZ_\psi}\int_{(3I)^c}\frac{1}{\dist(x,I)^d}\psi\Big(\frac{\ell(I)}{\dist(x,I)}\Big)\Norm{h_I}{1}\ud x \\
  &\lesssim\Norm{K}{CZ_0}\abs{I}\Norm{h_I}{\infty}+
    \Norm{K}{CZ_\psi}\int_{\ell(I)}^{\infty}\frac{1}{r^d}\psi\Big(\frac{\ell(I)}{r}\Big)r^{d-1}\ud r\Norm{h_I}{1} \\
  &=\Norm{K}{CZ_0}\abs{I}^{1/2}+\Norm{K}{CZ_\psi}\int_0^1\psi(t)\frac{\ud t}{t}\abs{I}^{1/2} \\
  &\lesssim\big(\Norm{K}{CZ_0}+\Norm{K}{CZ_\psi}\big)\abs{I}^{1/2}
\end{split}
\end{equation*}
by the Dini condition for $\psi$ in the last step.

Alternatively, the part giving the factor $\Norm{K}{CZ_0}$ could have been estimated by
\begin{equation*}
\begin{split}
  \int_{3I\setminus I}\Babs{\int K(x,y)h_I(y)\ud y}\ud x
  \leq\abs{3I\setminus I}^{1/2}\Norm{Th_I}{2}
  \lesssim\abs{I}^{1/2}\Norm{T}{L^2\to L^2}.
\end{split}
\end{equation*}

\subsubsection*{Case $\ell(I)< 2^{-r}\ell(J)$}
Since $I\subseteq J_I$ is good, we have
\begin{equation*}
  \dist(I,J_I^c)>\ell(J_I)\phi\Big(\frac{\ell(I)}{\ell(J_I)}\Big)\gtrsim\ell(J)\phi\Big(\frac{\ell(I)}{\ell(J)}\Big)
\end{equation*}
and hence
\begin{equation*}
\begin{split}
  \int_{J_I^c}\abs{Th_I(x)}\ud x
  &\lesssim\Norm{K}{CZ_\psi}\int_{J_I^c}\frac{1}{d(x,I)^d}\psi\Big(\frac{\ell(I)}{\dist(x,I)}\Big)\Norm{h_I}{1}\ud x \\
  &\lesssim\Norm{K}{CZ_\psi}\int_{\ell(J)\phi(\ell(I)/\ell(J))}\frac{1}{r^d}\psi\Big(\frac{\ell(I)}{r}\Big)r^{d-1}\ud r\cdot\Norm{h_I}{1}\\
  &=\Norm{K}{CZ_\psi}\int_0^{\ell(I)/\ell(J)\cdot\phi(\ell(I)/\ell(J))^{-1}}\psi(t)\frac{\ud t}{t}\cdot\abs{I}^{1/2}.\qedhere
\end{split}
\end{equation*}
\end{proof}

Now we can organize
\begin{equation*}
  \sigma_{\operatorname{in}}'
  :=\sum_J\sum_{I\subsetneq J}\pair{g}{h_J}\pair{1_{J_I^c}(h_J-\ave{h_J}_{J_I})}{Th_I}\pair{h_I}{f}
  =\sum_{i=1}^{\infty}\sum_J\sum_{\substack{I\subset J\\ \ell(I)=2^{-i}\ell(J)}},
\end{equation*}
and the inner sum is recognized as
\begin{equation*}
  \big(\Norm{K}{CZ_0}+\Norm{K}{CZ_\psi}\big)\Psi(2^{-i}\phi(2^{-i})^{-1})\pair{g}{A_J^{i0} f},
\end{equation*}
or with $\Norm{T}{L^2\to L^2}$ in place of $\Norm{K}{CZ_0}$,
for a cancellative averaging operator of type $(i,0)$.

On the other hand,
\begin{equation*}
\begin{split}
  \sigma_{\operatorname{in}}''
    &:=\sum_J\sum_{I\subsetneq J}\pair{g}{h_J}\ave{h_J}_{I}\pair{1}{Th_I}\pair{h_I}{f} \\
    &=\sum_I\Big\langle\sum_{J\supsetneq I}\pair{g}{h_J}h_J\Big\rangle_{I}\pair{1}{Th_I}\pair{h_I}{f} \\
    &=\sum_I\ave{g}_I\pair{T^*1}{h_I}\pair{h_I}{f}  \\
    &=\Bpair{\sum_I\ave{g}_I\pair{T^*1}{h_I}h_I}{f} =:\pair{\Pi_{T^*1}g}{f}
    =\pair{g}{\Pi_{T^*1}^* f}.
\end{split}
\end{equation*}
Here $\Pi_{T^*1}$ is the \emph{paraproduct}, a non-cancellative shift composed of the non-cancellative averaging operators
\begin{equation*}
  A_I g=\pair{T^*1}{h_I}\ave{g}_I h_I=\abs{I}^{-1/2}\pair{T^*1}{h_I}\cdot\pair{g}{h_I^0}h_I
\end{equation*}
of type $(0,0)$.

Summarizing, we have
\begin{equation*}
\begin{split}
  \sigma_{\operatorname{in}}
  &=\sigma_{\operatorname{in}}'+\sigma_{\operatorname{in}}'' \\
  &=\big(\Norm{K}{CZ_0}+\Norm{K}{CZ_\psi}\big)\sum_{i=1}^{\infty}\Psi(2^{-i}\phi(2^{-i})^{-1})\pair{g}{S^{i0}f}+\pair{\Pi_{T^*1}g}{f},
\end{split}
\end{equation*}
where $\Psi(t)=\int_0^t\psi(s)\ud s/s$, and $\Norm{K}{CZ_0}$ could be replaced by $\Norm{T}{L^2\to L^2}$. Note that if we wanted to write $\Pi_{T^*1}$ in terms of a shift with correct normalization, we should divide and multiply it by $\Norm{T^*1}{\BMO}$, thus getting a shift times the factor $\Norm{T^*1}{\BMO}\lesssim\Norm{T}{L^2}+\Norm{K}{CZ_\psi}$

\subsection{Near-by cubes, $\sigma_{=}$ and $\sigma_{\operatorname{near}}$}
We are left with the sums $\sigma_{=}$ of equal cubes $I=J$, as well as  $\sigma_{\operatorname{near}}$ of disjoint near-by cubes with $\dist(I,J)\leq\ell(J)\phi(\ell(I)/\ell(J))$.
Since $I$ is good, this necessarily implies that $\ell(I)>2^{-r}\ell(J)$. Then, for a given $J$, there are only boundedly many related $I$ in this sum.

\begin{lemma}
\begin{equation*}
  \abs{\pair{h_J}{Th_I}}\lesssim\Norm{K}{CZ_0}+\delta_{IJ}\Norm{T}{WBP}.
\end{equation*}
\end{lemma}

Note that if we used the $L^2$-boundedness of $T$ instead of the $CZ_0$ and $WBP$ conditions (as is done in Theorem~\ref{thm:formula}, we could also estimate simply
\begin{equation*}
  \abs{\pair{h_J}{Th_I}}\leq\Norm{h_J}{2}\Norm{T}{L^2\to L^2}\Norm{h_I}{2}=\Norm{T}{L^2\to L^2}.
\end{equation*}

\begin{proof}
For disjoint cubes, we estimate directly
\begin{equation*}
\begin{split}
  \abs{\pair{h_J}{Th_I}}
  &\lesssim\Norm{K}{CZ_0}\int_J\int_I\frac{1}{\abs{x-y}^d}\ud y\ud x\Norm{h_J}{\infty}\Norm{h_I}{\infty} \\
  &\leq\Norm{K}{CZ_0}\int_J\int_{3J\setminus J}\frac{1}{\abs{x-y}^d}\ud y\ud x\abs{J}^{-1/2}\abs{I}^{-1/2} \\
   &\lesssim\Norm{K}{CZ_0}\abs{J}\abs{J}^{-1/2}\abs{J}^{-1/2}=\Norm{K}{CZ_0},
\end{split}
\end{equation*}
since $\abs{I}\eqsim\abs{J}$.

For $J=I$, let $I_i$ be its dyadic children. Then
\begin{equation*}
\begin{split}
  \abs{\pair{h_I}{Th_I}}
  &\leq\sum_{i,j=1}^{2^d}\abs{\ave{h_I}_{I_i}\ave{h_I}_{I_j}\pair{1_{I_i}}{T1_{I_j}}} \\
  &\lesssim\Norm{K}{CZ_0}\sum_{j\neq i}\abs{I}^{-1}\int_{I_i}\int_{I_j}\frac{1}{\abs{x-y}^d}\ud x\ud y
    +\sum_i\abs{I}^{-1}\abs{\pair{1_{I_i}}{T1_{I_i}}} \\
   &\lesssim\Norm{K}{CZ_0}+\Norm{T}{WBP},
\end{split}
\end{equation*}
by the same estimate as earlier for the first term, and the weak boundedness property for the second.
\end{proof}

With this lemma, the sum $\sigma_{=}$ is recognized as a cancellative dyadic shift of type $(0,0)$ as such:
\begin{equation*}
\begin{split}
  \sigma_{=}
  &=\sum_{I\in\mathscr{D}}1_{\good}(I)\cdot\pair{g}{h_I}\pair{h_I}{Th_I}\pair{h_I}{f} \\
  &=\big(\Norm{K}{CZ_0}+\Norm{T}{WBP}\big)\pair{g}{S^{00}f},
\end{split}
\end{equation*}
where the factor in front could also be replaced by $\Norm{T}{L^2\to L^2}$.

For $I$ and $J$ participating in $\sigma_{\operatorname{near}}$, we conclude from Lemma~\ref{lem:IveeJ} that $K:=I\vee J$ satisfies $\ell(K)\leq 2^r\ell(I)$, and hence we may organize
\begin{equation*}
  \sigma_{\operatorname{near}}
  =\sum_{i=1}^r\sum_{j=1}^i \sum_K\sum_{\substack{I,J:I\vee J=K\\ \dist(I,J)\leq\ell(J)\phi(\ell(I)/\ell(J))\\ \ell(I)=2^{-i}\ell(K) \\ \ell(J)=2^{-j}\ell(K)}},
\end{equation*}
and the innermost sum is recognized as $\Norm{K}{CZ_0}\pair{g}{A^{ij}_K f}$ for some cancellative averaging operator of type $(i,j)$.

Summarizing, we have
\begin{equation*}
  \sigma_{=}+\sigma_{\operatorname{near}}
  =\big(\Norm{K}{CZ_0}+\Norm{T}{WBP}\big)\pair{g}{S^{00}f}
    +\Norm{K}{CZ_0}\sum_{j=1}^r\sum_{i=j}^r \pair{g}{S^{ij}f},
\end{equation*}
where $S^{00}$ and $S^{ij}$ are cancellative dyadic shifts, and the factor $\big(\Norm{K}{CZ_0}+\Norm{T}{WBP}\big)$ could also be replaced by $\Norm{T}{L^2\to L^2}$.

\subsection{Synthesis}
We have checked that
\begin{equation*}
\begin{split}
  \sum_{\ell(I)\leq\ell(J)} &1_{\operatorname{good}}(I)\pair{g}{h_J}\pair{h_J}{Th_I}\pair{h_I}{f} \\
  &=\big(\Norm{K}{CZ_0}+\Norm{K}{CZ_\psi}\big)\Big(\sum_{1\leq j\leq i<\infty}\phi(2^{-i})^{-d}\psi(2^{-i}\phi(2^{-i})^{-1})\pair{g}{S^{ij}f} \\
  &\qquad\qquad+\sum_{1\leq i<\infty}\Psi(2^{-i}\phi(2^{-i})^{-1}))\pair{g}{S^{i0}f}\Big) \\
  &\qquad+\big(\Norm{K}{CZ_0}+\Norm{T}{WBP}\big)\pair{g}{S^{00}f}+\pair{g}{\Pi_{T^*1}^*f}
\end{split}
\end{equation*}
where $\Psi(t)=\int_0^t\psi(s)\ud s/s$, $\Pi_{T^*1}$ is a paraproduct---a non-cancellative shift of type $(0,0)$--, and all other $S^{ij}$ is a cancellative dyadic shifts of type $(i,j)$.

By symmetry (just observing that the cubes of equal size contributed precisely to the presence of the cancellative shifts of type $(i,i)$, and that the dual of a shift of type $(i,j)$ is a shift of type $(j,i)$), it follows that
\begin{equation*}
\begin{split}
  \sum_{\ell(I)>\ell(J)} &1_{\operatorname{good}}(J)\pair{g}{h_J}\pair{h_J}{Th_I}\pair{h_I}{f} \\
  &=\big(\Norm{K}{CZ_0}+\Norm{K}{CZ_\psi}\big)\Big(\sum_{1\leq i<j<\infty}\phi(2^{-j})^{-d}\psi(2^{-j}\phi(2^{-j})^{-1})\pair{g}{S^{ij}f} \\
  &\qquad+\sum_{1\leq j<\infty}\Psi(2^{-j}\phi(2^{-j})^{-1}))\pair{g}{S^{0j}f}\Big)+\pair{g}{\Pi_{T1}f}
\end{split}
\end{equation*}
so that altogether
\begin{equation*}
\begin{split}
  \sum_{I,J} &1_{\operatorname{good}}(\min\{I,J\}) \pair{g}{h_J}\pair{h_J}{Th_I}\pair{h_I}{f} \\
  &=\big(\Norm{K}{CZ_0}+\Norm{K}{CZ_\psi}\big)\Big(\sum_{i=1}^{\infty}\Psi(2^{-i}\phi(2^{-i})^{-1}))(\pair{g}{S^{i0}f}+\pair{g}{S^{0i}f}) \\
   &\qquad\qquad+\sum_{i,j=1}^{\infty}\phi(2^{-\max(i,j)})^{-d}\psi(2^{-\max(i,j)}\phi(2^{-\max(i,j)})^{-1})\pair{g}{S^{ij}f}\Big)\\
   &\qquad+\big(\Norm{K}{CZ_0}+\Norm{T}{WBP}\big)\pair{g}{S^{00}f}+\pair{g}{\Pi_{T1}f}+\pair{g}{\Pi_{T^*1}^*f},
\end{split}
\end{equation*}
and this completes the proof of Theorem~\ref{thm:formula}.

\section{Two-weight theory for dyadic shifts}

Before proceeding further, it is convenient to introduce a useful trick due to E.~Sawyer. Let $\sigma$ be an everywhere positive, finitely-valued function. Then $f\in L^p(w)$ if and only if $\phi=f/\sigma\in L^p(\sigma^p w)$, and they have equal norms in the respective spaces. Hence an inequality
\begin{equation}\label{eq:original}
  \Norm{Tf}{L^p(w)}\leq N\Norm{f}{L^p(w)}\qquad\forall f\in L^p(w)
\end{equation}
is equivalent to
\begin{equation*}
  \Norm{T(\phi\sigma)}{L^p(w)}\leq N\Norm{\phi\sigma}{L^p(w)}=N\Norm{\phi}{L^p(\sigma^p w)}\qquad\forall \phi\in L^p(\sigma^p w).
\end{equation*}
This is true for any $\sigma$, and we now choose it in such a way that $\sigma^p w=\sigma$, i.e., $\sigma=w^{-1/(p-1)}=w^{1-p'}$, where $p'$ is the dual exponent. So finally \eqref{eq:original} is equivalent to
\begin{equation*}
  \Norm{T(\phi\sigma)}{L^p(w)}\leq N\Norm{\phi}{L^p(\sigma)}\qquad\forall \phi\in L^p(\sigma).
\end{equation*}
This formulation has the advantage that the norm on the right and the operator
\begin{equation*}
  T(\phi\sigma)(x)=\int K(x,y)\phi(y)\cdot\sigma(y)\ud y
\end{equation*}
involve integration with respect to the same measure $\sigma$. In particular, the $A_2$ theorem is equivalent to
\begin{equation*}
  \Norm{T(f\sigma)}{L^2(w)}\leq c_T[w]_{A_2}\Norm{f}{L^2(\sigma)}
\end{equation*}
for all $f\in L^2(w)$, for all $w\in A_2$ and $\sigma=w^{-1}$. But once we know this, we can also study this two-weight inequality on its own right, for two general measures $w$ and $\sigma$, which need not be related by the pointwise relation $\sigma(x)=1/w(x)$. 

\begin{theorem}\label{thm:2weight}
Let $\sigma$ and $w$ be two locally finite measures with
\begin{equation*}
  [w,\sigma]_{A_2}:=\sup_Q\frac{w(Q)\sigma(Q)}{\abs{Q}^2}<\infty.
\end{equation*}
Then a dyadic shift $S$ of type $(i,j)$ satisfies $S(\sigma\cdot):L^2(\sigma)\to L^2(w)$ if and only if
\begin{equation*}
  \mathfrak{S}:=\sup_Q\frac{\Norm{1_Q S(\sigma 1_Q)}{L^2(w)}}{\sigma(Q)^{1/2}},\qquad
  \mathfrak{S}^*:=\sup_Q\frac{\Norm{1_Q S^*(w 1_Q)}{L^2(\sigma)}}{w(Q)^{1/2}}
\end{equation*}
are finite, and in this case
\begin{equation*}
  \Norm{S(\sigma\cdot)}{L^2(\sigma)\to L^2(w)}
  \lesssim(1+\kappa)(\mathfrak{S}+\mathfrak{S}^*)+(1+\kappa)^2[w,\sigma]_{A_2}^{1/2},
\end{equation*}
where $\kappa=\max\{i,j\}$.
\end{theorem}

This result from my work with P\'erez, Treil, and Volberg \cite{HPTV} was preceded by an analogous qualitative version due to Nazarov, Treil, and Volberg \cite{NTV:2weightHaar}.

The proof depends on decomposing functions in the spaces $L^2(w)$ and $L^2(\sigma)$ in terms of expansions similar to the Haar expansion in $L^2(\R^d)$. Let $\D^{\sigma}_I$ be the orthogonal projection of $L^2(\sigma)$ onto its subspace of functions supported on $I$, constant on the subcubes of $I$, and with vanishing integral with respect to $\ud\sigma$. Then any two $\D_I^{\sigma}$ are orthogonal to each other. Under the additional assumption that the $\sigma$ measure of quadrants of $\R^d$ is infinite, we have the expansion
\begin{equation*}
  f=\sum_{Q\in\mathscr{D}}\D_Q^{\sigma}f
\end{equation*}
for all $f\in L^2(\sigma)$, and Pythagoras' theorem says that
\begin{equation*}
  \Norm{f}{L^2(\sigma)}=\Big(\sum_{Q\in\mathscr{D}}\Norm{\D_Q^{\sigma}f}{L^2(\sigma)}^2\Big)^{1/2}.
\end{equation*}
(These formulae needs a slight adjustment if the $\sigma$ measure of quadrants is finite; Theorem~\ref{thm:2weight} remains true without this extra assumption.) Let us also write
\begin{equation*}
  \D^{\sigma,i}_K:=\sum_{\substack{I\subseteq K\\ \ell(I)=2^{-i}\ell(K)}}\D^{\sigma}_I.
\end{equation*}
For a fixed $i\in\N$, these are also orthogonal to each other, and the above formulae generalize to
\begin{equation*}
  f=\sum_{Q\in\mathscr{D}}\D_Q^{\sigma,i}f,\qquad\Norm{f}{L^2(\sigma)}=\Big(\sum_{Q\in\mathscr{D}}\Norm{\D_Q^{\sigma,i}f}{L^2(\sigma)}^2\Big)^{1/2}.
\end{equation*}

The proof is in fact very similar in spirit to that of Theorem~\ref{thm:formula}; it is another $T1$ argument, but now with respect to the measures $\sigma$ and $w$ in place of the Lebesgue measure. We hence expand
\begin{equation*}
  \pair{g}{S(\sigma f)}_w
  =\sum_{Q,R\in\mathscr{D}}\pair{\D_R^w g}{S(\sigma\D_Q^{\sigma}f)}_w,\qquad f\in L^2(\sigma),\ g\in L^2(w),
\end{equation*}
and estimate the matrix coefficients
\begin{equation}\label{eq:2weightMatrix}
\begin{split}
   \pair{\D_R^{w}g}{S(\sigma \D^{\sigma}_Q f)}_{w}
   &=\sum_K\pair{\D^w_R g}{A_K(\sigma\D^{\sigma}_Q f)}_w \\
   &=\sum_K\sum_{I,J\subseteq K}a_{IJK}\pair{\D^w_R g}{h_J}_w\pair{h_I}{\D^{\sigma}_Q f}_{\sigma}.
\end{split}
\end{equation}
For $\pair{h_I}{\D^{\sigma}_Q f}_{\sigma}\neq 0$, there must hold $I\cap Q\neq\varnothing$, thus $I\subseteq Q$ or $Q\subsetneq I$. But in the latter case $h_I$ is constant on $Q$, while $\int\D^{\sigma}_Q f\cdot\sigma=0$, so the pairing vanishes even in this case. Thus the only nonzero contributions come from $I\subseteq Q$, and similarly from $J\subseteq R$. Since $I,J\subseteq K$, there holds
\begin{equation*}
  \big(I\subseteq Q\subsetneq K\quad\text{or}\quad K\subseteq Q\big)\qquad\text{and}\qquad\big(J\subseteq R\subsetneq K\quad\text{or}\quad K\subseteq R\big).
\end{equation*}

\subsection{Disjoint cubes}
Suppose now that $Q\cap R=\varnothing$, and let $K$ be among those cubes for which $A_K$ gives a nontrivial contribution above. Then it cannot be that $K\subseteq Q$, since this would imply that $Q\cap R\supseteq K\cap J=J\neq\varnothing$, and similarly it cannot be that $K\subseteq R$.  Thus $Q,R\subsetneq K$, and hence
\begin{equation*}
  Q\vee R\subseteq K.
\end{equation*}
Then
\begin{align*}
   \abs{\pair{\D_R^w g}{S(\sigma\D^{\sigma}_Q f)}_w}
   &\leq \sum_{K\supseteq Q\vee R}\abs{\pair{\D^w_R g}{A_K(\sigma\D^{\sigma}_Q f)}_{w}} \\
   &\lesssim \sum_{K\supseteq Q\vee R} \frac{\Norm{\D^w_R g}{L^1(w)} \Norm{\D^{\sigma}_Q f}{L^1(\sigma)}}{\abs{K}} \\
   &\lesssim \frac{\Norm{\D^w_R g}{L^1(w)} \Norm{\D^{\sigma}_Q f}{L^1(\sigma)}}{\abs{Q\vee R}}
\end{align*}
On the other hand, we have $Q\supseteq I$, $R\supseteq J$ for some $I,J\subseteq K$ with $\ell(I)=2^{-i}\ell(K)$ and $\ell(J)=2^{-j}\ell(K)$. Hence $2^{-i}\ell(K)\leq\ell(Q)$ and $2^{-j}\ell(K)\leq\ell(R)$, and thus
\begin{equation*}
  Q\vee R\subseteq K\subseteq Q^{(i)}\cap R^{(j)}.
\end{equation*}

Now it is possible to estimate the total contribution of the part of the matrix with $Q\cap R=\varnothing$. Let $P:=Q\vee R$ be a new auxiliary summation variable. Then $Q,R\subset P$, and $\ell(Q)=2^{-a}\ell(P)$, $\ell(R)=2^{-b}\ell(P)$ where $a=1,\ldots,i$, $b=1,\ldots,j$.
Thus
\begin{align*}
  \sum_{\substack{Q,R\in\mathscr{D}\\ Q\cap R=\varnothing}} &\abs{\pair{\D_R^{w}g}{S(\sigma\D^{\sigma}_Q f)}_{w}} \\
  &\lesssim\sum_{a=1}^i\sum_{b=1}^j\sum_{P\in\mathscr{D}}\frac{1}{\abs{P}}
  \sum_{\substack{Q,R\in\mathscr{D}:Q\vee R=P\\ \ell(Q)=2^{-a}\ell(P)\\ \ell(R)=2^{-b}\ell(P)}}\Norm{\D^w_R g}{L^1(\sigma)} \Norm{\D^{\sigma}_Q f}{L^1(w)} \\
  &\leq\sum_{a,b=1}^{i,j}\sum_{P\in\mathscr{D}}\frac{1}{\abs{P}}\sum_{\substack{R\subseteq P\\ \ell(R)=2^{-b}\ell(P)}}\Norm{\D^w_R g}{L^1(\sigma)}
        \sum_{\substack{Q\subseteq P\\ \ell(Q)=2^{-a}\ell(P)}}\Norm{\D^{\sigma}_{Q}f}{L^1(\sigma)} \\
  &=\sum_{a,b=1}^{i,j}\sum_{P\in\mathscr{D}}\frac{1}{\abs{P}}\BNorm{\sum_{\substack{R\subseteq P\\ \ell(R)=2^{-b}\ell(P)}}\D^w_R g}{L^1(\sigma)}
        \BNorm{\sum_{\substack{Q\subseteq P\\ \ell(Q)=2^{-a}\ell(P)}}\D^{\sigma}_{Q}f}{L^1(\sigma)}\\
   &\qquad\qquad\qquad\text{(by disjoint supports)} \\
  &=\sum_{a,b=1}^{i,j}\sum_{P\in\mathscr{D}}\frac{1}{\abs{P}}\Norm{\D^{w,j}_{P}g}{L^1(w)}\Norm{\D^{\sigma,i}_{P}f}{L^1(\sigma)} \\
  &\leq\sum_{a,b=1}^{i,j}\sum_{P\in\mathscr{D}}\frac{\sigma(P)^{1/2}w(P)^{1/2}}{\abs{P}}
     \Norm{\D^{w,j}_{P}g}{L^2(w)}\Norm{\D^{\sigma,i}_{P}f}{L^2(\sigma)}\\
  &\leq\sum_{a,b=1}^{i,j}[w,\sigma]_{A_2}^{1/2}\Big(\sum_{P\in\mathscr{D}}\Norm{\D^{w,j}_{P}g}{L^2(w)}^2\Big)^{1/2}
     \Big(\sum_{P\in\mathscr{D}}\Norm{\D^{\sigma,i}_{P}f}{L^2(\sigma)}^2\Big)^{1/2} \\
  &\leq ij[w,\sigma]_{A_2}^{1/2}\Norm{g}{L^2(w)}\Norm{f}{L^2(\sigma)}.
\end{align*}

\subsection{Deeply contained cubes}

Consider now the part of the sum with $Q\subset R$ and $\ell(Q)<2^{-i}\ell(R)$. (The part with $R\subset Q$ and $\ell(R)<2^{-j}\ell(Q)$ would be handled in a symmetrical manner.)

\begin{lemma}\label{lem:contCubesAlgebra}
For all $Q\subset R$ with $\ell(Q)<2^{-i}\ell(R)$, we have
\begin{equation*}
    \pair{\D^w_R g}{S(\sigma\D^{\sigma}_Q f)}_{w}
  =\ave{\D^w_R g}_{Q^{(i)}} \pair{S^*(w1_{Q^{(i)}})}{\D^{\sigma}_Q f}_{\sigma},
\end{equation*}
where further
\begin{equation*}
  \D^{\sigma}_Q S^*(w1_{Q^{(i)}})
  =\D^{\sigma}_Q S^*(w1_{P})\qquad\text{for any }P\supseteq Q^{(i)}.
\end{equation*}
\end{lemma}

Recall that $\D^{\sigma}_Q=(\D^{\sigma}_Q)^2=(\D^{\sigma}_Q)^*$ is an orthogonal projection on $L^2(\sigma)$, so that it can be moved to either or both sides of the pairing $\pair{\ }{\ }_{\sigma}$.

\begin{proof}
Recall formula~\eqref{eq:2weightMatrix}.
If $\pair{h_I}{\D_Q^{\sigma} f}_{\sigma}$ is nonzero, then $I\subseteq Q$, and hence
\begin{equation*}
  J\subseteq K=I^{(i)}\subseteq Q^{(i)}\subsetneq R
\end{equation*}
for all $J$ participating in the same $A_K$ as $I$. Thus $\D^w_R g$ is constant on $Q^{(i)}$, hence
\begin{equation*}
\begin{split}
  \pair{\D^w_R g}{A_K(\sigma\D^{\sigma}_Q f)}_{w}
  &=\pair{1_{Q^{(i)}}\D^w_R g}{A_K(\sigma\D^{\sigma}_Q f)}_{w} \\
   &=\ave{\D^w_R g}_{Q^{(i)}}^w \pair{1_{Q^{(i)}}}{A_K(\sigma\D^{\sigma}_Q f)}_{w} \\
   &=\ave{\D^w_R g}_{Q^{(i)}}^w \pair{A_K^*(w1_{Q^{(i)}})}{\D^{\sigma}_Q f}_{\sigma}.
\end{split}
\end{equation*}

Moreover, for any $P\supseteq Q^{(i)}\supseteq K$,
\begin{equation*}
\begin{split}
   \pair{\D^{\sigma}_Q A_K^*(w1_{Q^{(i)}})}{f}_{\sigma}
  &=\pair{1_{Q^{(i)}}}{A_K(\sigma\D^{\sigma}_Q f)}_{w} \\
  &=\int A_K(\sigma\D^{\sigma}_Q f)w \\
  &=\pair{1_P}{A_K(\sigma\D^{\sigma}_Q f)}_w
  =\pair{\D^{\sigma}_QA_K^*(w1_P)}{ f}_{\sigma}.
\end{split}
\end{equation*}
Summing these equalities over all relevant $K$, and using $S=\sum_K A_K$, gives the claim.
\end{proof}

By the lemma, we can then manipulate
\begin{align*}
  \sum_{\substack{Q,R:Q\subset R\\ \ell(Q)<2^{-i}\ell(R)}} &\pair{\D^w_R g}{S(\sigma\D^{\sigma}_Q f)}_{w} \\
  &=\sum_Q\Big(\sum_{R\supsetneq Q^{(i)}}\ave{\D^w_R g}_{Q^{(i)}}^w\Big)\pair{S^*(w1_{Q^{(i)}})}{\D^{\sigma}_Q f}_{\sigma} \\
  &=\sum_Q\ave{g}^{w}_{Q^{(i)}}\pair{S^*(w1_{Q^{(i)}})}{\D^{\sigma}_Q f}_{\sigma} \\
  &=\sum_R\ave{g}^w_R\Bpair{S^*(w1_R)}{\sum_{\substack{Q\subseteq R \\ \ell(Q)=2^{-i}\ell(R)}}\D^{\sigma}_Q f}_{\sigma} \\
  &=\sum_R\ave{g}^w_R\Bpair{S^*(w 1_R)}{\D^{\sigma,i}_R f}_{\sigma},
\end{align*}
where $\ave{g}^w_R:=w(R)^{-1}\int_R g\cdot w$ is the average of $g$ on $R$ with respect to the $w$ measure.

By using the properties of the pairwise orthogonal projections $\D^{\sigma,i}_R$ on $L^2(\sigma)$, 
the above series may be estimated as follows:
\begin{align*}
  &\Babs{\sum_{\substack{Q,R:Q\subset R\\ \ell(Q)<2^{-i}\ell(R)}} \pair{\D^w_R g}{S(\sigma\D^{\sigma}_Q f)}_{w}} \\
  &\leq\sum_R\abs{\ave{g}^w_R}\Norm{\D^{\sigma,i}_R S^*(w1_R)}{L^2(\sigma)}\Norm{\D^{\sigma,i}_R f}{L^2(\sigma)} \\
  &\leq\Big(\sum_R\abs{\ave{g}^w_R}^2\Norm{\D^{\sigma,i}_R S^*(w1_R)}{L^2(\sigma)}^2\Big)^{1/2}
     \Big(\sum_R\Norm{\D^{\sigma,i}_R f}{L^2(\sigma)}^2\Big)^{1/2},
\end{align*}
where the last factor is equal to $\Norm{f}{L^2(w)}$.

The first factor on the right is handled by the dyadic Carleson embedding theorem: It follows from the second equality of Lemma~\ref{lem:contCubesAlgebra}, namely $\D^{\sigma}_Q S^*(w1_Q^{(i)})=\D^{\sigma}_Q S^*(w1_P)$ for all $P\supseteq Q^{(i)}$, that $\D^{\sigma,i}_R S^*(w1_R)=\D^{\sigma}_Q S^*(w1_P)$ for all $P\subseteq R$. Hence, we have
\begin{equation*}
\begin{split}
  \sum_{R\subseteq P}\Norm{\D^{\sigma,i}_R S^*(w 1_R)}{L^2(\sigma)}^2
  &=\sum_{R\subseteq P}\Norm{\D^{\sigma,i}_R(1_P S^*(w 1_P))}{L^2(\sigma)}^2 \\
  &\leq\Norm{1_P S^*(w 1_P)}{L^2(\sigma)}^2\lesssim \mathfrak{S}_*^2\sigma(P)
\end{split}
\end{equation*}
by the (dual) testing estimate for the dyadic shifts. By the Carleson embedding theorem, it then follows that
\begin{equation*}
   \Big(\sum_R\abs{\ave{g}^w_R}^2 \Norm{\D^{\sigma,i}_R S^*(w 1_R)}{L^2(\sigma)}^2\Big)^{1/2}
   \lesssim \mathfrak{S}_*\Norm{g}{L^2(\sigma)},
\end{equation*}
and the estimation of the deeply contained cubes is finished.

\subsection{Contained cubes of comparable size}

It remains to estimate
\begin{equation*}
  \sum_{\substack{Q,R:Q\subseteq R\\ \ell(Q)\geq 2^{-i}\ell(R)}}\pair{\D^w_R g}{S(\sigma\D^{\sigma}_Q f)}_{w};
\end{equation*}
the sum over $R\subsetneq Q$ with $\ell(R)\geq 2^{-j}\ell(Q)$ would be handled in a symmetric manner. The sum of interest may be written as
\begin{equation*}
  \sum_{a=0}^i\sum_R\sum_{\substack{Q\subseteq R\\ \ell(Q)=2^{-a}\ell(R)}}\pair{\D^w_R g}{S(\sigma\D^{\sigma}_Q f)}_{w}
  =\sum_{a=0}^i\sum_R\pair{\D^w_R g}{S(\sigma\D^{\sigma,i}_R f)}_{w},
\end{equation*}
and
\begin{equation*}
  \pair{\D^w_R g}{S(\sigma\D^{\sigma,i}_R f)}_{w}
  =\sum_{k=1}^{2^d}\ave{\D^w_R g}_{R_k}\pair{S^*(w 1_{R_k})}{\D^{\sigma,i}_R f}_{\sigma}
\end{equation*}
where the $R_k$ are the $2^d$ dyadic children of $R$, and $\ave{\D^w_R g}_{R_k}$ is the constant valued of $\D^w_R g$ on $R_k$. Now
\begin{equation*}
  \pair{S^*(w 1_{R_k})}{\D^{\sigma,i}_R f}_{\sigma}
  =\pair{1_{R_k}S^*(w 1_{R_k})}{\D^{\sigma,i}_R f}_{\sigma}+\pair{S^*(w 1_{R_k})}{1_{R_k^c}\D^{\sigma,i}_R f}_{\sigma},
\end{equation*}
where
\begin{equation*}
  \abs{\pair{1_{R_k}S^*(w 1_{R_k})}{\D^{\sigma,i}_R f}_{\sigma}}
  \leq\mathfrak{S}_* w(R_k)^{1/2}\Norm{\D^{\sigma,i}_R f}{L^2(\sigma)}
\end{equation*}
and, observing that only those $A_K^*$ where $K$ intersects both $R_k$ and $R_k^c$ contribute to the second part,
\begin{equation*}
\begin{split}
  \abs{\pair{S^*(w 1_{R_k})}{1_{R_k^c}\D^{\sigma,i}_R f}_{\sigma}}
  &=\Babs{\sum_{K\supsetneq R_k}\pair{A_K^*(w 1_{R_k})}{1_{R_k^c}\D^{\sigma,i}_R f}_{\sigma}} \\
  &\lesssim\sum_{K\supseteq R}\frac{1}{\abs{K}}w(R_k)\Norm{\D^{\sigma,i}_R f}{L^1(\sigma)} \\
  &\lesssim\frac{1}{\abs{R}}w(R_k)\sigma(R)^{1/2}\Norm{\D^{\sigma,i}_R f}{L^1(\sigma)} \\
  &\leq\frac{w(R)^{1/2}\sigma(R)^{1/2}}{\abs{R}}w(R_k)^{1/2}\Norm{\D^{\sigma,i}_R f}{L^2(\sigma)} \\
  &\leq[w,\sigma]_{A_2}w(R_k)^{1/2}\Norm{\D^{\sigma,i}_R f}{L^2(\sigma)}.
\end{split}
\end{equation*}
It follows that
\begin{equation*}
  \abs{\pair{S^*(w 1_{R_k})}{\D^{\sigma,i}_R f}_{\sigma}}
  \lesssim(\mathfrak{S}_* +[w,\sigma]_{A_2})w(R_k)^{1/2}\Norm{\D^{\sigma,i}_R f}{L^2(\sigma)}
\end{equation*}
and hence
\begin{equation*}
  \abs{\pair{\D_R^w g}{S(\sigma\D^{\sigma,i}_R f)}_{w}}
  \lesssim(\mathfrak{S}_* +[w,\sigma]_{A_2})\Norm{\D^w_R g}{L^2(w)}\Norm{\D^{\sigma,i}_R f}{L^2(\sigma)}
 \end{equation*}
Finally,
\begin{align*}
  &\sum_{a=0}^i\sum_R \abs{\pair{\D_R^w g}{S(\sigma\D^{\sigma,i}_R f)}_{w}} \\
  &\lesssim  (\mathfrak{S}_* +[w,\sigma]_{A_2})\sum_{a=0}^i\Big(\sum_R\Norm{\D^w_R g}{L^2(\sigma)}^2\Big)^{1/2}
       \Big(\sum_R\Norm{\D^{\sigma,i}_R f}{L^2(\sigma)}^2\Big)^{1/2} \\
  &\leq(1+i)(\mathfrak{S}_* +[w,\sigma]_{A_2})\Norm{g}{L^2(w)}\Norm{f}{L^2(\sigma)}.
\end{align*}
The symmetric case with $R\subset Q$ with $\ell(R)\geq 2^{-j}\ell(Q)$ similarly yields the factor $(1+j)(\mathfrak{S} +[w,\sigma]_{A_2})$. This completes the proof of Theorem~\ref{thm:2weight}.

\section{Final decompositions: verification of the testing conditions}

We now turn to the estimation of the testing constant
\begin{equation*}
  \mathfrak{S}:=\sup_{Q\in\mathscr{D}}\frac{\Norm{1_Q S(\sigma 1_Q)}{L^2(w)}}{\sigma(Q)^{1/2}}.
\end{equation*}
Bounding $\mathfrak{S}_*$ is analogous by exchanging the roles of $w$ and $\sigma$. 

\subsection{Several splittings}
First observe that
\begin{equation*}
  1_Q S(\sigma 1_Q)
  =1_Q\sum_{K:K\cap Q\neq\varnothing}A_K(\sigma 1_Q)
  =1_Q\sum_{K\subseteq Q}A_K(\sigma 1_Q)+1_Q\sum_{K\supsetneq Q}A_K(\sigma 1_Q).
\end{equation*}
The second part is immediate to estimate even pointwise by
\begin{equation*}
   \abs{1_Q A_K(\sigma 1_Q)}\leq 1_Q\frac{\sigma(Q)}{\abs{K}},\qquad\sum_{K\supsetneq Q}\frac{1}{\abs{K}}\leq\frac{1}{\abs{Q}},
\end{equation*}
and hence its $L^2(w)$ norm is bounded by
\begin{equation*}
   \BNorm{1_Q\frac{\sigma(Q)}{\abs{Q}}}{L^2(w)}
   =\frac{w(Q)^{1/2}\sigma(Q)}{\abs{Q}}\leq[w,\sigma]_{A_2}\sigma(Q)^{1/2}.
\end{equation*}
So it remains to concentrate on $K\subseteq Q$, and we perform several consecutive splittings of this collection of cubes. First, we \textbf{separate scales} by introducing the splitting according to the $\kappa+1$ possible values of $\log_2\ell(K)\mod(\kappa+1)$. We denote a generic choice of such a collection by
\begin{equation*}
  \mathscr{K}=\mathscr{K}_k:=\{K\subseteq Q:\log_2\ell(K)\equiv k\mod(\kappa+1)\},
\end{equation*}
where $k$ is arbitrary but fixed. (We will drop the subscript $k$, since its value plays no role in the subsequent argument.) Next, we \textbf{freeze the $A_2$ characteristic} by setting
\begin{equation*}
  \mathscr{K}^a:=\Big\{K\in\mathscr{K}: 2^{a-1}<\frac{w(K)\sigma(K)}{\abs{K}}\leq 2^a\Big\},\qquad a\in\Z,\quad a\leq\ceil{\log_2[w,\sigma]_{A_2}},
\end{equation*}
where $\ceil{\ }$ means rounding up to the next integer.

In the next step, we \textbf{choose the principal cubes} $P\in\mathscr{P}^a\subseteq\mathscr{K}^a$.
This construction was first introduced by B. Muckenhoupt and R. Wheeden \cite{MW:77}, and it has been influential ever since.
Let $\mathscr{P}^a_0$ consist of all maximal cubes in $\mathscr{K}^a$, and inductively $\mathscr{P}^a_{p+1}$ consist of all maximal $P'\in\mathscr{K}^a$ such that
\begin{equation*}
  P'\subset P\in\mathscr{P}^a_p,\qquad \frac{\sigma(P')}{\abs{P'}}>2\frac{\sigma(P)}{\abs{P}}.
\end{equation*}
Finally, let $\mathscr{P}^a:=\bigcup_{p=0}^{\infty}\mathscr{P}^a_p$.
For each $K\in\mathscr{K}^a$, let $\Pi^a(K)$ denote the minimal $P\in\mathscr{P}^a$ such that $K\subseteq P$. Then we set
\begin{equation*}
  \mathscr{K}^a(P):=\{K\in\mathscr{K}^a:\Pi^a(K)=P\},\qquad P\in\mathscr{P}^a.
\end{equation*}

Note that $\sigma(K)/\abs{K}\leq 2\sigma(P)/\abs{P}$ for all $K\in\mathscr{K}^a(P)$, which allows us to \textbf{freeze the $\sigma$-to-Lebesgue measure ratio} by the final subcollections
\begin{equation*}
  \mathscr{K}^a_b(P):=\Big\{K\in\mathscr{K}^a(P): 2^{-b}<\frac{\sigma(K)}{\abs{K}}\frac{\abs{P}}{\sigma(P)}\leq 2^{1-b}\Big\},\qquad b\in\N.
\end{equation*}
We have
\begin{equation*}
\begin{split}
  &\{K\in\mathscr{D}:K\subseteq Q\}
  =\bigcup_{k=0}^{\kappa}\mathscr{K}_k,\qquad
  \mathscr{K}_k=\mathscr{K}=\bigcup_{a\leq\ceil{\log_2[w,\sigma]_{A_2}}}\mathscr{K}^a,\\
  &\mathscr{K}^a=\bigcup_{P\in\mathscr{P}^a}\mathscr{K}^a(P),\qquad
  \mathscr{K}^a(P)=\bigcup_{b=0}^{\infty}\mathscr{K}^a_b(P),\qquad
\end{split}
\end{equation*}
where all unions are disjoint. Note that we drop the reference to the separation-of-scales parameter $k$, since this plays no role in the forthcoming arguments. Recalling the notation for subshifts $S_{\mathscr{Q}}=\sum_{K\in\mathscr{Q}}A_K$, this splitting of collections of cubes leads to the splitting of the function
\begin{equation*}
  \sum_{K\subseteq Q}A_K(\sigma 1_Q)=\sum_{k=0}^{\kappa}\sum_{a\leq\ceil{\log_2[w,\sigma]_{A_2}}}\sum_{P\in\mathscr{P}^a}\sum_{b=0}^{\infty}S_{\mathscr{K}^a_b(P)}(\sigma 1_Q).
\end{equation*}
On the level of the function, we split one more time to write
\begin{equation*}
\begin{split}
  S_{\mathscr{K}^a_b(P)}(\sigma 1_Q)
  &=\sum_{n=0}^{\infty} 1_{E^a_b(P,n)}S_{\mathscr{K}^a_b(P)}(\sigma 1_Q),\\
  E^a_b(P,n) &:=\{x\in\R^d:n2^{-b}\ave{\sigma}_P<\abs{S_{\mathscr{K}^a_b(P)}(\sigma 1_Q)(x)}\leq (n+1)2^{-b}\ave{\sigma}_P\}.
\end{split}
\end{equation*}
This final splitting, from \cite{HLMORSU}, is not strictly `necessary' in that it was not part of the original argument in \cite{Hytonen:A2}, nor its predecessor in \cite{LPR}, which made instead more careful use of the cubes where $S_{\mathscr{K}^a_b(P)}(\sigma 1_Q)$ stays constant; however, it now seems that this splitting provides another simplification of the argument.

Now all relevant cancellation is inside the functions $S_{\mathscr{K}^a_b(P)}(\sigma 1_Q)$, so that we can simply estimate by the triangle inequality:
\begin{equation*}
\begin{split}
  &\Babs{\sum_{K\subseteq Q}A_K(\sigma 1_Q)} \\
  &\quad\leq\sum_{k=0}^{\kappa}\sum_{a}\sum_{P\in\mathscr{P}^a}\sum_{b=0}^{\infty}\sum_{n=0}^{\infty}(1+n)2^{-b}
     \ave{\sigma}_P 1_{\{\abs{S_{\mathscr{K}^a_b(P)}(\sigma 1_Q)}>n2^{-b}\ave{\sigma}_P\}},
\end{split}
\end{equation*}
and
\begin{equation*}
\begin{split}
  &\BNorm{\sum_{K\subseteq Q}A_K(\sigma 1_Q)}{L^2(w)} \\
  &\leq\sum_{k=0}^{\kappa}\sum_{a}\sum_{b=0}^{\infty}2^{-b}\sum_{n=0}^{\infty}(1+n)
     \BNorm{\sum_{P\in\mathscr{P}^a}\ave{\sigma}_P 1_{\{\abs{S_{\mathscr{K}^a_b(P)}(\sigma 1_Q)}>n2^{-b}\ave{\sigma}_P\}}}{L^2(w)}.
\end{split}
\end{equation*}
Obviously, we will need good estimates to be able to sum up these infinite series.

Write the last norm as
\begin{equation*}
   \Big(\int\Big[\sum_{P\in\mathscr{P}^a}\ave{\sigma}_P 1_{\{\abs{S_{\mathscr{K}^a_b(P)}(\sigma 1_Q)}>n2^{-b}\ave{\sigma}_P\}}(x)\Big]^2 \ud w(x)\Big)^{1/2},
\end{equation*}
observe that
\begin{equation*}
  \{\abs{S_{\mathscr{K}^a_b(P)}(\sigma 1_Q)}>n2^{-b}\ave{\sigma}_P\}\subseteq P,
\end{equation*}
and look at the integrand at a fixed point $x\in\R^d$. At this point we sum over a subset of those values of $\ave{\sigma}_P$ where the principal cube $P\owns x$. Let $P_0$ be the smallest cube such that $\abs{S_{\mathscr{K}^a_b(P)}}>n2^{-b}\ave{\sigma}_P$, let $P_1$ be the next smallest, and so on. Then $\ave{\sigma}_{P_m}<2^{-1}\ave{\sigma}_{P_{m-1}}<\ldots<2^{-m}\ave{\sigma}_{P_0}$ by the construction of the principal cubes, and hence
\begin{equation*}
\begin{split}
  \Big[\sum_{P\in\mathscr{P}^a}\ave{\sigma}_P 1_{\{\abs{S_{\mathscr{K}^a_b(P)}}>n2^{-b}\ave{\sigma}_P\}}(x)\Big]^2
  &=\Big[\sum_{m=0}^{\infty}\ave{\sigma}_{P_m}\Big]^2 \\
  &\leq\Big[\sum_{m=0}^{\infty}2^{-m}\ave{\sigma}_{P_0}\Big]^2
    =4\ave{\sigma}_{P_0}^2 \\
   &\leq 4\sum_{P\in\mathscr{P}^a}\ave{\sigma}_P^2 1_{\{\abs{S_{\mathscr{K}^a_b(P)}(\sigma 1_Q)}>n2^{-b}\ave{\sigma}_P\}}(x)
\end{split}
\end{equation*}
Hence
\begin{equation*}
\begin{split}
  &\BNorm{\sum_{P\in\mathscr{P}^a}\ave{\sigma}_P 1_{\{\abs{S_{\mathscr{K}^a_b(P)}(\sigma 1_Q)}>n2^{-b}\ave{\sigma}_P\}}}{L^2(w)} \\
  &\leq \Big(\int\Big[4\sum_{P\in\mathscr{P}^a}\ave{\sigma}_P^2 1_{\{\abs{S_{\mathscr{K}^a_b(P)}(\sigma 1_Q)}>n2^{-b}\ave{\sigma}_P\}}\Big]w\Big)^{1/2} \\
  &=2\Big(\sum_{P\in\mathscr{P}^a}\ave{\sigma}_P^2 w(\{\abs{S_{\mathscr{K}^a_b(P)}(\sigma 1_Q)}>n2^{-b}\ave{\sigma}_P\})\Big)^{1/2},
\end{split}
\end{equation*}
and it remains to obtain good estimates for the measure of the level sets
\begin{equation*}
  \{\abs{S_{\mathscr{K}^a_b(P)}(\sigma 1_Q)}>n2^{-b}\ave{\sigma}_P\}.
\end{equation*}

\subsection{Weak-type and John--Nirenberg-style estimates}
We still need to estimate the sets above. Recall that $S_{\mathscr{K}^a_b(P)}$ is a subshift of $S$, which in particular has its scales separated so that $\log_2\ell(K)\equiv k\mod (\kappa+1)$ for all $K$ for which $A_K$ participating in $S_{\mathscr{K}^a_b(P)}$ is nonzero and $k\in\{0,1,\ldots,\kappa:=\max\{i,j\}\}$ is fixed, $S$ being of type $(i,j)$. The following estimate deals with such subshifts, which we simply denote by $S$.

\begin{proposition}
Let $S$ be a dyadic shift of type $(i,j)$ with scales separated. Then
\begin{equation*}
  \abs{\{\abs{Sf}>\lambda\}}\leq\frac{C}{\lambda}\Norm{f}{1},\qquad\forall\lambda>0,
\end{equation*}
where $C$ depends only on the dimension.
\end{proposition}

\begin{proof}
The proof uses the classical Calder\'on--Zygmund decomposition:
\begin{equation*}
  f=g+b,\qquad b:=\sum_{L\in\mathscr{B}}b_L:=\sum_{L\in \mathscr{B}}1_B(f-\ave{f}_L),
\end{equation*}
where $L\in\mathscr{B}$ are the maximal dyadic cubes with $\ave{|f|}_L>\lambda$; hence $\ave{|f|}_L\leq 2^d\lambda$. As usual,
\begin{equation*}
  g=f-b=1_{\big(\bigcup\mathscr{B}\big)^c}f+\sum_{L\in\mathscr{B}}\ave{f}_L
\end{equation*}
satisfies $\Norm{g}{\infty}\leq 2^d\lambda$ and $\Norm{g}{1}\leq\Norm{f}{1}$, hence $\Norm{g}{2}^2\leq\Norm{g}{\infty}\Norm{g}{1}\leq 2^{d}\lambda\Norm{f}{1}$, and thus
\begin{equation*}
  \abs{\{\abs{Sg}>\tfrac12\lambda\}}\leq\frac{4}{\lambda^2}\Norm{Sg}{2}^2
  \leq\frac{4}{\lambda^2}\Norm{g}{2}^2\leq 4\cdot 2^d\frac{1}{\lambda}\Norm{f}{1}.
\end{equation*}

It remains to estimate $\{\abs{Sb}>\tfrac12\lambda\}$. First observe that
\begin{equation*}
  Sb=\sum_{K\in\mathscr{D}}\sum_{L\in\mathscr{B}}A_K b_L
   =\sum_{L\in\mathscr{B}}\Big(\sum_{K\subseteq L}A_K b_L+\sum_{K\supsetneq L}A_K b_L\Big),
\end{equation*}
since $A_K b_L\neq 0$ only if $K\cap L\neq\varnothing$. Now
\begin{equation*}
\begin{split}
  \abs{\{\abs{Sb}>\tfrac12\lambda\}}
  &\leq\Babs{\Big\{\Babs{\sum_{L\in\mathscr{B}}\sum_{K\subseteq L}A_K b_L}>0\Big\}}
    +\Babs{\Big\{\Babs{\sum_{L\in\mathscr{B}}\sum_{K\supsetneq L}A_K b_L}>\tfrac12\lambda\Big\}} \\
  &\leq\sum_{L\in\mathscr{B}}\abs{L}+\frac{2}{\lambda}\BNorm{\sum_{L\in\mathscr{B}}\sum_{K\supsetneq L}A_K b_L}{1} \\
  &\leq\frac{1}{\lambda}\Norm{f}{1}+\frac{2}{\lambda}\sum_{L\in\mathscr{B}}\sum_{K\supsetneq L}\Norm{A_K b_L}{1},
\end{split}
\end{equation*}
where we used the elementary properties of the Calder\'on--Zygmund decomposition to estimate the first term.

For the remaining double sum, we still need some observations. Recall that
\begin{equation*}
  A_K b_L=\sum_{\substack{I,J\subseteq K \\ \ell(I)=2^{-i}\ell(K)\\ \ell(J)=2^{-j}\ell(K)}}a_{IJK}h_I\pair{h_J}{b_L}.
\end{equation*}
Now, if $\ell(K)>2^{\kappa}\ell(L)\geq 2^j\ell(L)$, then $\ell(J)>\ell(L)$, and hence $h_J$ is constant on $L$. But the integral of $b_L$ vanishes, hence $\pair{h_J}{b_L}=0$ for all relevant $J$, and thus $A_K b_L=0$ whenever $\ell(K)>2^\kappa\ell(L)$.

Thus, in the inner sum, the only possible nonzero terms are $A_K b_L$ for $K=L^{(m)}$ for $m=1,\ldots,\kappa$. By the separation of scales, at most one of these terms is nonzero, and we write $\tilde{L}$ for the corresponding unique $K$. So in fact
\begin{equation*}
  \frac{2}{\lambda}\sum_{L\in\mathscr{B}}\sum_{K\supsetneq L}\Norm{A_K b_L}{1}
  =\frac{2}{\lambda}\sum_{L\in\mathscr{B}}\Norm{A_{\tilde L} b_L}{1}
  \leq\frac{2}{\lambda}\sum_{L\in\mathscr{B}}\Norm{b_L}{1}
  \leq\frac{2}{\lambda}\cdot 2\Norm{f}{1}=\frac{4}{\lambda}\Norm{f}{1}
\end{equation*}
by using the normalized boundedness of the averaging operators $A_{\tilde L}$ on $L^1(\R^d)$, and an elementary estimate for the bad part of the Calder\'on--Zygmund decomposition.

Altogether, we obtain the claim with $C=4\cdot 2^d+5$.
\end{proof}

For the special subshifts $S_{\mathscr{K}^a_b(P)}$, we can improve the weak-type $(1,1)$ estimate to an exponential decay:

\begin{proposition}
Let $S_{\mathscr{K}^a_b(P)}$ be the subshift of $S$ as constructed earlier. Then the following estimate holds when $\nu$ is either the Lebesgue measure or $w$:
\begin{equation*}
  \nu\Big(\Big\{\abs{S_{\mathscr{K}^a_b(P)}(\sigma1_Q)}>C2^{-b}\ave{\sigma}_P\cdot t\Big\}\Big)
  \lesssim C2^{-t}\nu(P),\qquad t\geq 0,
\end{equation*}
where $C$ is a constant.
\end{proposition}

\begin{proof}
Let $\lambda:=C2^{-b}\ave{\sigma}_P$, where $C$ is a large constant, and $n\in\Z_+$. Let $x\in\R^d$ be a point where
\begin{equation}\label{eq:>nLambda}
  \abs{S_{\mathscr{K}^a_b(P)}(\sigma 1_Q)(x)}>n\lambda.
\end{equation}
Then for all small enough $L\in\mathscr{K}^a_b(P)$ with $L\owns x$, there holds
\begin{equation*}
  \Babs{\sum_{\substack{K\in\mathscr{K}^a_b(P) \\ K\supseteq L}}A_K(\sigma 1_Q)(x)}>n\lambda.
\end{equation*}
Since $\displaystyle\sum_{\substack{K\in\mathscr{K}^a_b(P) \\ K\supsetneq L}}A_K(\sigma1_Q)$ is constant on $L$ (thanks to separation of scales), and
\begin{equation}\label{eq:ALwQ}
  \Norm{A_L(\sigma 1_Q)}{\infty}\lesssim\frac{\sigma(L)}{\abs{L}}\leq 2^{1-b}\frac{\sigma(P)}{\abs{P}},
\end{equation}
it follows that
\begin{equation}\label{eq:>n-2/3}
   \Babs{\sum_{\substack{K\in\mathscr{K}^a_b(P) \\ K\supsetneq L}}A_K(\sigma 1_Q)}>(n-\tfrac{2}{3})\lambda\qquad\text{on }L.
\end{equation}
Let $\mathscr{L}\subseteq\mathscr{K}^a_b(P)$ be the collection of maximal cubes with the above property. Thus all $L\in\mathscr{L}$ are disjoint, and all $x$ with \eqref{eq:>nLambda} belong to some $L$. By maximality of $L$, the minimal $L^*\in\mathscr{K}^a_b(S)$ with $L^*\supsetneq L$ satisfies
\begin{equation*}
   \Babs{\sum_{\substack{K\in\mathscr{K}^a_b(P) \\ K\supsetneq L^*}}A_K(\sigma 1_Q)}\leq(n-\tfrac{2}{3})\lambda\qquad\text{on }L^*.
\end{equation*}
By an estimate similar to \eqref{eq:ALwQ}, with $L^*$ in place of $L$, it follows that
\begin{equation*}
  \Babs{\sum_{\substack{K\in\mathscr{K}^a_b(P) \\ K\supsetneq L}}A_K(\sigma 1_Q)}\leq (n-\tfrac{1}{3})\lambda\qquad\text{on }L.
\end{equation*}
Thus, if $x$ satisfies \eqref{eq:>nLambda} and $x\in L\in\mathscr{L}$, then necessarily
\begin{equation*}
  \abs{S_{\{K\in\mathscr{K}^a_b(P); K\subseteq L\}}(\sigma 1_{Q})(x)}=
  \Babs{\sum_{\substack{K\in\mathscr{K}^a_b(P) \\ K\subseteq L}}A_K(\sigma 1_Q)(x)}>\tfrac13\lambda.
\end{equation*}
Using the weak-type $L^1$ estimate to the shift $S_{\{K\in\mathscr{K}^a_b(P);K\subseteq L\}}$ of type $(i,j)$ with scales separated, noting that $A_K(\sigma 1_Q)=A_K(\sigma 1_L)$ for $K\subseteq L$, it follows that
\begin{align*}
  \Babs{\Big\{\Babs{\sum_{\substack{K\in\mathscr{K}^a_b(P) \\ K\subseteq L}}A_K(\sigma 1_Q)(x)}>\tfrac13\lambda\Big\}}
  &\leq \frac{C}{\lambda}\sigma(L) \\
  &\leq\frac{C}{\lambda}2^{1-b}\frac{\sigma(S\cap Q)}{\abs{S}}\abs{L} \leq \tfrac13\abs{L},
\end{align*}
provided that the constant in the definition of $\lambda$ was chosen large enough.
Recalling \eqref{eq:>n-2/3}, there holds
\begin{align*}
  \Babs{\sum_{K\in\mathscr{K}^a_b(P)}A_K(\sigma 1_Q)}
  &\geq\Babs{\sum_{\substack{K\in\mathscr{K}^a_b(P) \\ K\supsetneq L}}A_K(\sigma1_Q)}
     -\Babs{\sum_{\substack{K\in\mathscr{K}^a_b(P) \\ K\subseteq L}}A_K(\sigma 1_Q)} \\
  &>(n-\tfrac23)\lambda-\tfrac13\lambda=(n-1)\lambda\quad\text{on }\tilde{L}\subset L\text{ with }\abs{\tilde{L}}\geq\tfrac23\abs{L}.
\end{align*}
Thus
\begin{align*}
  \abs{\{\abs{S_{\mathscr{K}^a_b(P)}(\sigma 1_Q)}>n\lambda\}}
  &\leq\sum_{L\in\mathscr{L}}\abs{L\cap \{\abs{S_{\mathscr{K}^a_b(P)}(\sigma1_Q)}>n\lambda\}} \\
  &\leq\sum_{L\in\mathscr{L}}\abs{\{\abs{S_{\{K\in\mathscr{K}^a_b(P):K\subseteq L\}}(\sigma 1_Q)}>\tfrac13\lambda\}} \\
  &\leq\sum_{L\in\mathscr{L}}\tfrac13\abs{L}\leq\sum_{L\in\mathscr{L}}\tfrac13\cdot\tfrac 32\abs{\tilde{L}} \\
  &\leq\tfrac12\sum_{L\in\mathscr{L}} \abs{L\cap\{ \abs{S_{\mathscr{K}^a_b(P)}(\sigma 1_Q)}>(n-1)\lambda\}} \\
  &\leq\tfrac12\abs{\{ \abs{S_{\mathscr{K}^a_b(P)}(\sigma 1_Q)}>(n-1)\lambda\}}.
\end{align*}
By induction it follows that
\begin{align*}
  \abs{\{\abs{S_{\mathscr{K}^a_b(P)}(\sigma 1_Q)}>n\lambda\}}
  &\leq 2^{-n}\abs{\{ \abs{S_{\mathscr{K}^a_b(P)}(\sigma 1_Q)}>0\}} \\
  &\leq 2^{-n}\sum_{M\in\mathscr{M}}\abs{M}\leq 2^{-n}\abs{P},
\end{align*}
where $\mathscr{M}$ is the collection of maximal cubes in $\mathscr{K}^a_b(S)$.

Recalling that we defined $\lambda:=C2^{-b}\ave{\sigma}_P$ in the beginning of the proof, the previous display gives precisely the claim of the Proposition in the case that $\nu$ is the Lebesgue measure. We still need to consider the case that $\nu=w$. To this end, selected intermediate steps of the above computation, as well as the definition of $\mathscr{K}^a_b(P)$, will be exploited. Recall that $K\in\mathscr{K}^a$ means that $2^{a-1}<\ave{w}_K\ave{\sigma}_K\leq 2^a$, while $K\in\mathscr{K}^a_b(P)$ means that in addition $2^{-b}<\ave{\sigma}_K/\ave{\sigma}_P\leq 2^{1-b}$. Put together, this says that
\begin{equation*}
  2^{a+b-2}\ave{\sigma}_P<\frac{w(K)}{\abs{K}}<2^{a+b}\ave{\sigma}_P\qquad\forall K\in\mathscr{K}^a_b(P).
\end{equation*}
Hence, using the collections $\mathscr{L},\mathscr{M}\subseteq\mathscr{K}^a_b(P)$ as above,
\begin{align*}
  w(\{\abs{S_{\mathscr{K}^a_b(P)}(\sigma 1_Q)}>n\lambda\})
  &\leq\sum_{L\in\mathscr{L}}w(L) 
    \leq\sum_{L\in\mathscr{L}}2^{a+b}\ave{\sigma}_P\abs{L} \\
  &\leq 2^{a+b}\ave{\sigma}_P\abs{\{ \abs{S_{\mathscr{K}^a_b(P)}(\sigma 1_Q)}>(n-1)\lambda\}} \\
  &\leq 2^{a+b}\ave{\sigma}_P\cdot 2^{-n}\sum_{M\in\mathscr{M}}\abs{M} \\
  &\leq 4\cdot 2^{-n}\sum_{M\in\mathscr{M}}w(M)\leq 4\cdot 2^{-n}w(S).\qedhere
\end{align*}
\end{proof}

\subsection{Conclusion of the estimation of the testing conditions}
Recall that
\begin{equation*}
\begin{split}
  &\BNorm{\sum_{K\subseteq Q}A_K(\sigma 1_Q)}{L^2(w)} \\
  &\leq\sum_{k=0}^{\kappa}\sum_{a}\sum_{b=0}^{\infty}2^{-b}\sum_{n=0}^{\infty}(1+n)
     \BNorm{\sum_{P\in\mathscr{P}^a}\ave{\sigma}_P 1_{\{\abs{S_{\mathscr{K}^a_b(P)}(\sigma 1_Q)}>n2^{-b}\ave{\sigma}_P\}}}{L^2(w)}
\end{split}
\end{equation*}
and
\begin{equation*}
\begin{split}
    &  \BNorm{\sum_{P\in\mathscr{P}^a}\ave{\sigma}_P 1_{\{\abs{S_{\mathscr{K}^a_b(P)}(\sigma 1_Q)}>n2^{-b}\ave{\sigma}_P\}}}{L^2(w)} \\
   &\leq  2\Big(\sum_{P\in\mathscr{P}^a}\ave{\sigma}_P^2 w(\{\abs{S_{\mathscr{K}^a_b(P)}(\sigma 1_Q)}>n2^{-b}\ave{\sigma}_P\})\Big)^{1/2} \\
   &\leq C\Big(\sum_{P\in\mathscr{P}^a}\ave{\sigma}_P^2 2^{-n/C}w(P)\Big)^{1/2} \\
   &=C2^{-cn}\Big(\sum_{P\in\mathscr{P}^a}\frac{\sigma(P)w(P)}{\abs{P}^2}\sigma(P)\Big)^{1/2} \\
   &\leq C2^{-cn}\Big(2^a\sum_{P\in\mathscr{P}^a}\sigma(P)\Big)^{1/2},
\end{split}
\end{equation*}
recalling the freezing of the $A_2$ characteristic between $2^{a-1}$ and $2^a$ for cubes in $\mathscr{K}^a\supseteq\mathscr{P}^a$.

For the summation over the principal cubes, we observe that
\begin{equation*}
\begin{split}
  \sum_{P\in\mathscr{P}^a}\sigma(P)
  =\sum_{P\in\mathscr{P}^a}\ave{\sigma}_P\abs{P}
  =\int_Q\sum_{P\in\mathscr{P}^a}\ave{\sigma}_P 1_P(x)\ud x.
\end{split}
\end{equation*}
At any given $x$, if $P_0\subsetneq P_1\subsetneq\ldots\subseteq Q$ are the principal cubes containing it, we have
\begin{equation*}
  \sum_{P\in\mathscr{P}^a}\ave{\sigma}_P 1_P(x)
  =\sum_{m=0}^{\infty}\ave{\sigma}_{P_m}
  \leq\sum_{m=0}^{\infty}2^{-m}\ave{\sigma}_{P_0}=2\ave{\sigma}_{P_0}
  \leq 2M(\sigma 1_Q)(x),
\end{equation*}
where $M$ is the dyadic maximal operator. Hence
\begin{equation*}
  \sum_{P\in\mathscr{P}^a}\sigma(P)
  \leq 2\int_Q M(\sigma 1_Q)\ud x
  \leq 2[\sigma]_{A_\infty}\sigma(Q),
\end{equation*}
where we use the following notion of the $A_\infty$ characteristic:
\begin{equation*}
  [\sigma]_{A_\infty}:=\sup_Q\frac{1}{\sigma(Q)}\int_Q M(\sigma 1_Q)\ud x;
\end{equation*}
this was implicit already in the work of Fujii \cite{Fujii:weightedBMO} and it was taken as an explicit definition by the author and C. P\'erez \cite{HytPer}.

Substituting back, we have
\begin{equation*}
\begin{split}
  &\BNorm{\sum_{K\subseteq Q}A_K(\sigma 1_Q)}{L^2(w)} \\
  &\leq\sum_{k=0}^{\kappa}\sum_{a}\sum_{b=0}^{\infty}2^{-b}\sum_{n=0}^{\infty}(1+n)
     \BNorm{\sum_{P\in\mathscr{P}^a}\ave{\sigma}_P 1_{\{\abs{S_{\mathscr{K}^a_b(P)}(\sigma 1_Q)}>n2^{-b}\ave{\sigma}_P\}}}{L^2(w)} \\
   &\leq\sum_{k=0}^{\kappa}\sum_{a}\sum_{b=0}^{\infty}2^{-b}\sum_{n=0}^{\infty}(1+n)\cdot
      C2^{-cn}\Big(2^a\sum_{P\in\mathscr{P}^a}\sigma(P)\Big)^{1/2} \\
   &\leq\sum_{k=0}^{\kappa}\sum_{a}\sum_{b=0}^{\infty}2^{-b}\sum_{n=0}^{\infty}(1+n)\cdot
      C2^{-cn}\big(2^a[\sigma]_{A_\infty}\big)^{1/2} \\
   &=C\cdot[\sigma]_{A_\infty}^{1/2}\sum_{k=0}^{\kappa}\Big(\sum_{a\leq\ceil{\log_2[w,\sigma]_{A_2}}}2^{a/2}\Big)
       \Big(\sum_{b=0}^{\infty}2^{-b}\Big)\Big(\sum_{n=0}^{\infty}(1+n)\cdot 2^{-cn}\Big) \\
    &\leq C\cdot[\sigma]_{A_\infty}^{1/2}\cdot(1+\kappa)\cdot [w,\sigma]_{A_2}^{1/2},
\end{split}
\end{equation*}
and thus the testing constant $\mathfrak{S}$ is estimated by
\begin{equation*}
  \mathfrak{S}\leq C\cdot(1+\kappa)\cdot[w,\sigma]_{A_2}^{1/2}\cdot[\sigma]_{A_\infty}^{1/2}.
\end{equation*}
By symmetry, exchanging the roles of $w$ and $\sigma$, we also have the analogous result for $\mathfrak{S}^*$, and so we have completed the proof of the following:

\begin{theorem}\label{thm:testing}
Let $\sigma,w\in A_\infty$ be functions which satisfy the joint $A_2$ condition
\begin{equation*}
  [w,\sigma]_{A_2}:=\sup_Q\frac{w(Q)\sigma(Q)}{\abs{Q}^2}<\infty.
\end{equation*}
Then the testing constants $\mathfrak{S}$ and $\mathfrak{S}^*$ associated with a dyadic shift $S$ of type $(i,j)$ satisfy the following bounds, where $\kappa:=\max\{i,j\}$:
\begin{equation*}
\begin{split}
   \mathfrak{S} &\leq C\cdot(1+\kappa)\cdot[w,\sigma]_{A_2}^{1/2}\cdot[\sigma]_{A_\infty}^{1/2}, \\
  \mathfrak{S}^* &\leq C\cdot(1+\kappa)\cdot[w,\sigma]_{A_2}^{1/2}\cdot[w]_{A_\infty}^{1/2}. 
\end{split}
\end{equation*}
\end{theorem}

\section{Conclusions}

In this section we simply collect the fruits of the hard work done above.
A combination of Theorems~\ref{thm:2weight} and \ref{thm:testing} gives the following two-weight inequality, whose qualitative version was pointed out by Lacey, Petermichl and Reguera \cite{LPR}. In the precise form as stated, this result and its consequences below were obtained by P\'erez and myself \cite{HytPer}, although originally formulated only in the case that $\sigma=w^{-1}$ is the dual weight.

\begin{theorem}\label{thm:2weightShift}
Let $\sigma,w\in A_\infty$ be functions which satisfy the joint $A_2$ condition
\begin{equation*}
  [w,\sigma]_{A_2}:=\sup_Q\frac{w(Q)\sigma(Q)}{\abs{Q}^2}<\infty.
\end{equation*}
Then a dyadic shift $S$ of type $(i,j)$ satisfies $S(\sigma\cdot):L^2(\sigma)\to L^2(w)$, and more precisely
\begin{equation*}
  \Norm{S(\sigma\cdot)}{L^2(\sigma)\to L^2(w)}
  \lesssim (1+\kappa)^2[w,\sigma]_{A_2}^{1/2}\big([w]_{A_\infty}^{1/2}+[\sigma]_{A_\infty}^{1/2}\big),
\end{equation*}
where $\kappa=\max\{i,j\}$.
\end{theorem}

The quantitative bound as stated, including the polynomial dependence on $\kappa$, allows to sum up these estimates in the Dyadic Representation Theorem to deduce:

\begin{theorem}\label{thm:2weightCZO}
Let $\sigma,w\in A_\infty$ be functions which satisfy the joint $A_2$ condition.
Then any $L^2$ bounded Calder\'on--Zygmund operator $T$ whose kernel $K$ has H\"older type modulus of continuity $\psi(t)=t^{\alpha}$, $\alpha\in(0,1)$, satisfies
\begin{equation*}
  \Norm{T(\sigma\cdot)}{L^2(\sigma)\to L^2(w)}
  \lesssim (\Norm{T}{L^2\to L^2}+\Norm{K}{CZ_\alpha})[w,\sigma]_{A_2}^{1/2}\big([w]_{A_\infty}^{1/2}+[\sigma]_{A_\infty}^{1/2}\big).
\end{equation*}
\end{theorem}

Recalling the dual weight trick and specializing to the one-weight situation with $\sigma=w^{-1}$, this in turn gives:

\begin{theorem}\label{thm:1weightCZO}
Let $w\in A_2$.
Then any $L^2$ bounded Calder\'on--Zygmund operator $T$ whose kernel $K$ has H\"older type modulus of continuity $\psi(t)=t^{\alpha}$, $\alpha\in(0,1)$, satisfies
\begin{equation*}
\begin{split}
  \Norm{T}{L^2(w)\to L^2(w)}
  &\lesssim (\Norm{T}{L^2\to L^2}+\Norm{K}{CZ_\alpha})[w]_{A_2}^{1/2}\big([w]_{A_\infty}^{1/2}+[w^{-1}]_{A_\infty}^{1/2}\big) \\
 &\lesssim (\Norm{T}{L^2\to L^2}+\Norm{K}{CZ_\alpha})[w]_{A_2}.
\end{split}
\end{equation*}
\end{theorem}

The second displayed line is the original $A_2$ theorem \cite{Hytonen:A2}, and it follows from the first line by $[w]_{A_\infty}\lesssim[w]_{A_2}$ and $[w^{-1}]_{A_\infty}\lesssim [w^{-1}]_{A_2}=[w]_{A_2}$ (see Lemma~\ref{lem:A2Ainf} below). Its strengthening on the first line was first observed in my joint work with C.~P\'erez \cite{HytPer}. Note that, compared to the introductory statement in Theorem~\ref{thm:A2}, the dependence on the operator $T$ has been made more explicit. (The implied constants in the notation ``$\lesssim$'' only depend on the dimension and the H\"older exponent $\alpha$.) This dependence on $\Norm{T}{L^2\to L^2}$ and $\Norm{K}{CZ_\alpha}$ is implicit in the original proof. 

For completeness, we include the proof (in the stated form essentially from \cite{LPR}, but see also \cite{HytPer} for more general comparison of $A_\infty$ and $A_p$ constants) that

\begin{lemma}\label{lem:A2Ainf}
For all weights $w\in A_2$, we have
\begin{equation*}
  [w]_{A_\infty}:=\sup_Q\frac{1}{w(Q)}\int_Q M(1_Q w)\ud x\leq 8[w]_{A_2}.
\end{equation*}
\end{lemma}

\begin{proof}
Let $\mathcal{P}$ be the principal cubes of Muckenhoupt and Wheeden \cite{MW:77} given by $\mathcal{P}=\bigcup_{p=0}^\infty\mathcal{P}_p$, where $\mathcal{P}_0:=\{Q\}$ and $\mathcal{P}_{p+1}$ consists of the maximal $P'\subset P\in\mathcal{P}_p$ with $w(P')/\abs{P'}>2w(P)/\abs{P}$. Then
\begin{equation*}
  M(1_Q w)(x)=\sup_{R:x\in R\subseteq Q}\frac{w(R)}{\abs{R}}\leq 2\sup_{P\in\mathcal{P}:x\in P}\frac{w(P)}{\abs{P}}\leq 2\sum_{P\in\mathcal{P}}\frac{w(P)}{\abs{P}}1_P(x),
\end{equation*}
and hence
\begin{equation*}
  \int_Q M(1_Q w)\ud x\leq 2\sum_{P\in\mathcal{P}}w(P).
\end{equation*}
Consider the pairwise disjoint sets $E(P):=P\setminus\bigcup_{P'\in\mathcal{P}:P'\subsetneq P}P'$. Since
\begin{equation*}
  \sum_{\substack{P'\subsetneq P\\ P'\text{ maximal}}}\abs{P'}
  \leq\sum_{\substack{P'\subsetneq P\\ P'\text{ maximal}}} \frac{w(P')\abs{P}}{2w(P)}\leq \frac{w(P)\abs{P}}{2w(P)}=\frac{\abs{P}}{2},
\end{equation*}
it follows that $\abs{E(P)}\geq\frac12\abs{P}$. We derive a similar condition for the weighted measure from the $A_2$ condition. Indeed,
\begin{equation*}
\begin{split}
  \abs{E(P)}
  &=\int_{E(P)}w^{1/2}w^{-1/2}\ud x
  \leq\Big(\int_{E(P)}w\ud x\Big)^{1/2}\Big(\int_P w^{-1}\ud x\Big)^{1/2} \\
  &=w(E(P))^{1/2}\Big(\fint_P w^{-1}\ud x\Big)^{1/2}\abs{P}^{1/2} \\
  &\leq w(E(P))^{1/2}[w]_{A_2}^{1/2}\Big(\fint_P w\ud x\Big)^{-1/2}\abs{P}^{1/2}
    =\Big([w]_{A_2}\frac{w(E(P))}{w(P)}\Big)^{1/2}\abs{P}.
\end{split}
\end{equation*}
Using $\abs{P}\leq 2\abs{E(P)}$ and squaring, this shows that
\begin{equation*}
  w(P)\leq 4[w]_{A_2}w(E(P)).
\end{equation*}
After this, it is immediate to compute that
\begin{equation*}
  \sum_{P\in\mathcal{P}}w(P)
  \leq 4[w]_{A_2}\sum_{P\in\mathcal{P}}w(E(P))
  \leq 4[w]_{A_2}w(Q),
\end{equation*}
since the sets $E(P)$ are pairwise disjoint and contained in $Q$.
\end{proof}

\section{Further results and remarks}

This final section briefly collects, without proofs, some further related developments, and poses some open problems.

The $A_2$ theorem implies a corresponding $A_p$ theorem for all $p\in(1,\infty)$. This follows from a version of the celebrated extrapolation theorem, one of the most useful tools in the theory of $A_p$ weights. The extrapolation theorem was first found by J. L. Rubio de Francia \cite{Rubio:factorAp}, and shortly after (so soon that it was published earlier) another proof was given by J. Garc{\'{\i}}a-Cuerva \cite{Garcia:extrapolation}. For the present purposes, we need a quantitative form of the extrapolation theorem, which is due to Dragi\v{c}evi\'c, Grafakos, Pereyra, and Petermichl \cite{DGPP}, and reads as follows. Although relatively recent, it was known well before the proof of the full $A_2$ theorem.

\begin{theorem}\label{thm:extrap}
If an operator $T$ satisfies
\begin{equation*}
  \Norm{T}{L^2(w)\to L^2(w)}\leq C_T [w]_{A_2}^{\tau}
\end{equation*}
for all $w\in A_2$, then it satisfies
\begin{equation*}
  \Norm{T}{L^p(w)\to L^p(w)}\leq c_p C_T [w]_{A_p}^{\tau\max\{1,1/(p-1)\}}
\end{equation*}
for all $p\in(1,\infty)$ and $w\in A_p$.
\end{theorem}

\begin{corollary}\label{cor:Ap}
Let $p\in(1,\infty)$ and $w\in A_p$.
Then any $L^2$ bounded Calder\'on--Zygmund operator $T$ whose kernel $K$ has H\"older type modulus of continuity $\psi(t)=t^{\alpha}$, $\alpha\in(0,1)$, satisfies
\begin{equation*}
  \Norm{T}{L^p(w)\to L^p(w)}
 \lesssim (\Norm{T}{L^2\to L^2}+\Norm{K}{CZ_\alpha})[w]_{A_p}^{\max\{1,1/(p-1)\}}.
\end{equation*}
\end{corollary}

It is also possible to apply a version of the extrapolation argument to the mixed $A_2$/$A_\infty$ bounds \cite{HytPer}, but this did not give the optimal results for $p\neq 2$. However, by setting up a different argument directly in $L^p(w)$, the following bounds were obtained in my collaboration with M.~Lacey \cite{HytLac}:

\begin{theorem}
Let $p\in(1,\infty)$ and $w\in A_p$.
Then any $L^2$ bounded Calder\'on--Zygmund operator $T$ whose kernel $K$ has H\"older type modulus of continuity $\psi(t)=t^{\alpha}$, $\alpha\in(0,1)$, satisfies
\begin{equation*}
  \Norm{T}{L^p(w)\to L^p(w)}
  \lesssim (\Norm{T}{L^2\to L^2}+\Norm{K}{CZ_\alpha})[w]_{A_p}^{1/p}\big([w]_{A_\infty}^{1/p'}+[w^{1-p'}]_{A_\infty}^{1/p}\big).
\end{equation*}
\end{theorem}

For weak-type bounds, which were investigated by Lacey, Martikainen, Orponen, Reguera, Sawyer, Uriarte-Tuero, and myself \cite{HLMORSU}, we need only `half' of the strong-type upper bound:

\begin{theorem}
Let $p\in(1,\infty)$ and $w\in A_p$.
Then any $L^2$ bounded Calder\'on--Zygmund operator $T$ whose kernel $K$ has H\"older type modulus of continuity $\psi(t)=t^{\alpha}$, $\alpha\in(0,1)$, satisfies
\begin{equation*}
\begin{split}
  \Norm{T}{L^p(w)\to L^{p,\infty}(w)}
  &\lesssim (\Norm{T}{L^2\to L^2}+\Norm{K}{CZ_\alpha})[w]_{A_p}^{1/p}[w]_{A_\infty}^{1/p'} \\
  &\lesssim (\Norm{T}{L^2\to L^2}+\Norm{K}{CZ_\alpha})[w]_{A_p}.
\end{split}
\end{equation*}
\end{theorem}

All these results remain valid for the non-linear operators given by the \emph{maximal truncations}
\begin{equation*}
  T_{\natural}f(x):=\sup_{\varepsilon>0}\abs{T_{\varepsilon}f(x)},\qquad
  T_{\varepsilon}f(x):=\int_{\abs{x-y}>\varepsilon}K(x,y)f(y)\ud y,
\end{equation*}
which have been addressed in \cite{HytLac,HLMORSU}. In \cite{HLMORSU} it was also shown that the sharp weighted bounds for dyadic shifts can be made linear (instead of quadratic) in $\kappa$, a result recovered by a different (Bellman function) method by Treil \cite{Treil:linear}. Earlier polynomial-in-$\kappa$ Bellman function estimates for the shifts were due to Nazarov and Volberg \cite{NV}. An extension of the $A_2$ theorem to abstract metric spaces with a doubling measure (spaces of homogeneous type) is due to Nazarov, Reznikov, and Volberg \cite{NRV}.

A higher degree of non-linearity is obtained by replacing the supremum over $\epsilon>0$ defining the maximal truncation by one of the \emph{variation norms}
\begin{equation*}
  \Norm{\{v_\epsilon\}_{\epsilon>0}}{V^q}
  :=\sup_{\{\epsilon_i\}_{i\in\Z}}\Big(\sum_i\abs{v_{\epsilon_i}-v_{\epsilon_{i+1}}}^q\Big)^{1/q},
\end{equation*}
where the supremum is over all monotone sequences $\{\epsilon_i\}_{i\in\Z}\subset(0,\infty)$. Sharp weighted bounds for the $q$-variation ($q\in(2,\infty)$) of Calder\'on--Zygmund operators were first proved by Hyt\"onen--Lacey--P\'erez \cite{HLP}, although replacing the sharp truncation $T_\epsilon f(x)$ by a smooth truncation
\begin{equation*}
  T^\phi_\epsilon f(x):=\int \phi\Big(\frac{\abs{x-y}}{\epsilon}\Big)K(x,y)f(y)\ud y,
\end{equation*}
where $\phi$ is smooth and $0\leq\phi\leq 1_{(1,\infty)}$. Sharp weighted bounds for the $q$-variation of the sharp truncations with $\phi=1_{(1,\infty)}$ were recently obtained by de Fran\c{c}a Silva and Zorin-Kranich \cite{FSZK}.

The approach to the $q$-variation in \cite{HLP} was through a non-probabilistic counterpart of the Dyadic Representation, a Dyadic Domination, which was independently discovered by Lerner \cite{Lerner:domination,Lerner:simple}. Another advantage of this method was its ability to handle Calder\'on--Zygmund kernels with weaker moduli of continuity $\psi$ than those treated by the present approach; namely any moduli $\psi$ subject to the log-bumped Dini condition $\int_0^1\psi(t)(1+\log\frac1t)\frac{\ud t}{t}<\infty$.

In its original form, the Dyadic Domination theorem gave a domination in norm, which improved to pointwise domination by Conde-Alonso and Rey \cite{CondeRey} and, independently, by Lerner and Nazarov \cite{LerNaz:book}. All these approaches required the same log-Dini condition,  and the necessity of the logarithmic correction to the Dini-condition remained open for some time, until it was finally eliminated by Lacey~\cite{Lacey:elem} by yet another approach. The following quantitative form of Lacey's theorem was obtained by L. Roncal, O. Tapiola and the author \cite{HytRoncal}, and with a simpler proof by Lerner \cite{Lerner:simplest}:

\begin{theorem}\label{thm:logDini}
Let $w\in A_2$. Then any $L^2$ bounded Calder\'on--Zygmund operator $T$ whose kernel $K$ has modulus of continuity $\psi$ , satisfies
\begin{equation*}
  \Norm{T}{L^2(w)\to L^2(w)}
 \lesssim \Big(\Norm{T}{L^2\to L^2}+\Norm{K}{CZ_0}+\Norm{K}{CZ_\psi}\int_0^1\psi(t)\frac{\ud t}{t}\Big)[w]_{A_2}.
\end{equation*}
\end{theorem}

Asking for even less regularity, one may wonder about the sharp weighted bound for the class of rough homogeneous singular integral operators
\begin{equation*}
  Tf(x)=\text{p.v.}\int_{\R^d}\frac{\Omega(y)}{\abs{y}^d}f(x-y)\ud y,
\end{equation*}
where
\begin{equation*}
  \Omega(y)=\Omega(\frac{y}{\abs{y}}),\qquad\Omega\in L^\infty(\mathbb{S}^{d-1}),\qquad\int_{\mathbb{S}^{d-1}}\Omega(\sigma)\ud\sigma=0.
\end{equation*}
Their qualitative boundedness $T:L^2(w)\to L^2(w)$ is known for $w\in A_2$ (see Watson~\cite{Watson}). Roncal, Tapiola and the author \cite{HytRoncal} showed that $\Norm{T}{L^2(w)\to L^2(w)}\lesssim \Norm{\Omega}{\infty}[w]_{A_2}^2$, but it is not known whether this quadratic dependence on $[w]_{A_2}$ is sharp.

\subsection{The Beurling operator and its powers}\label{sec:Beurling}

One of the key original motivations to study the $A_2$ theorem was a conjecture of Astala--Iwaniec--Saksman \cite{AIS} concerning the special case where $T$ is the Beurling operator
\begin{equation*}
  Bf(z):=-\frac{1}{\pi}\operatorname{p.v.}\int_{\C}\frac{1}{\zeta^2}f(z-\zeta)\ud A(\zeta),
\end{equation*}
and $A$ is the area measure (two-dimensional Lebesgue measure) on $\C\simeq\R^2$. This was the first Calder\'on--Zygmund operator for which the $A_2$ theorem was proven; it was achieved by Petermichl and Volberg \cite{PV}, confirming the mentioned conjecture of Astala, Iwaniec, and Saksman \cite{AIS}. Another proof of the $A_2$ theorem for this specific operator is due to Dragi\v{c}evi\'c and Volberg \cite{DV}.

The powers $B^n$ of $B$ have also been studied, and then it is of interest to understand the growth of the norms as a function of $n$. Shortly before the proof of the full $A_2$ theorem, by methods specific to the Beurling operator, O.~Dragi\v{c}evi\'c \cite{Dragicevic:cubic} was able to prove the cubic growth
\begin{equation*}
    \Norm{B^n}{L^2(w)\to L^2(w)}\lesssim\abs{n}^{3}[w]_{A_2},\qquad n\in\Z\setminus\{0\}.
\end{equation*}

Now, let us see what the general $A_2$ theorem gives for these specific powers.
It is known (see e.g. \cite{DPV}) that $B^n$ is the convolution operator with the kernel
\begin{equation*}
  K_n(z)=(-1)^n\frac{\abs{n}}{\pi}\Big(\frac{\bar{z}}{z}\Big)^n\abs{z}^{-2},
\end{equation*}
and it is elementary to check that this satisfies $\Norm{K_n}{CZ_\alpha}\lesssim\abs{n}^{1+\alpha}$ for any $\alpha\in(0,1)$. Moreover, since $B$ is an isometry on $L^2(\C)$, we have $\Norm{B^n}{L^2\to L^2}=1$. From Theorem~\ref{thm:1weightCZO} we deduce:

\begin{corollary}
The powers $B^n$ of the Beurling operator satisfy
\begin{equation*}
  \Norm{B^n}{L^2(w)\to L^2(w)}\lesssim\abs{n}^{1+\alpha}[w]_{A_2},\qquad\alpha>0,
\end{equation*}
where the implied constant depends on $\alpha$.
\end{corollary}

A sharper estimate still is provided by Theorem~\ref{thm:logDini}, as observed in \cite{HytRoncal}:

\begin{corollary}
The powers $B^n$ of the Beurling operator satisfy
\begin{equation*}
  \Norm{B^n}{L^2(w)\to L^2(w)}\lesssim\abs{n}(1+\log\abs{n}) [w]_{A_2}.
\end{equation*}
\end{corollary}

For this it suffices to check that, defining the modulus of continuity
\begin{equation*}
 \psi_n(t):=\min\{\abs{n}t,1\},
\end{equation*}
we have $\Norm{K_n}{CZ_{\psi_n}}\lesssim\abs{n}$ and hence
\begin{equation*}
  \Norm{K_n}{CZ_{\psi_n}}\int_0^1\psi_n(t)\frac{\ud t}{t}\lesssim \abs{n}(1+\log\abs{n}).
\end{equation*}
However, a better bound would follow if we had the $A_2$ theorem for the rough singular integrals in the form
\begin{equation*}
  \Norm{T}{L^2(w)\to L^2(w)}\lesssim \Norm{\Omega}{\infty}[w]_{A_2},
\end{equation*}
for this would lead to the linear estimate $\Norm{B^n}{L^2(w)\to L^2(w)}\lesssim\abs{n}[w]_{A_2}$, simply by viewing the kernels $K_n$ (although smooth), as rough kernels of homogeneous singular integrals.

Let us notice that no bound better than this is possible, at least on the scale of power-type dependence on $\abs{n}$:

\begin{proposition}
No bound of the form $\Norm{B^n}{L^2(w)\to L^2(w)}\lesssim\abs{n}^{1-\epsilon}[w]_{A_2}^{\tau}$ can be valid for any $\epsilon,\tau>0$.
\end{proposition}

\begin{proof}
Suppose for contradiction that such a bound holds for some fixed $\epsilon,\tau>0$ and all $n\in\Z\setminus\{0\}$. By Theorem~\ref{thm:extrap}, we deduce that
\begin{equation*}
  \Norm{B^n}{L^p(w)\to L^p(w)}\lesssim_p\abs{n}^{1-\epsilon}[w]_{A_p}^{\tau\max\{1,1/(p-1)\}},
\end{equation*}
and hence in particular we have the unweighted bound
\begin{equation*}
  \Norm{B^n}{L^p\to L^p}\lesssim_p\abs{n}^{1-\epsilon},\qquad 1<p<\infty.
\end{equation*}
However, it has been shown by Dragi\v{c}evi\'c, Petermichl and Volberg that the correct dependence here is
\begin{equation*}
  \Norm{B^n}{L^p\to L^p}\eqsim_p\abs{n}^{\abs{1-2/p}},\qquad 1<p<\infty.
\end{equation*}
The previous two displays are clearly in contradiction for $p$ close to either $1$ or~$\infty$, and we are done.
\end{proof}

The quest for the $A_2$ theorem began from the investigations of the Beurling transform, but clearly even this case is not yet fully understood.

\section*{Acknowledgements}
The author was supported by the European Union via the ERC Starting Grant ``Analytic--probabilistic methods for borderline singular integrals'', and by the Academy of Finland via grants 130166 and 133264 and the Centre of Excellence in Analysis and Dynamics Research. The core of this exposition was first presented as a series of lectures during the Doc-Course ``Harmonic analysis, metric spaces and applications to P.D.E.''  at Universidad de Sevilla, Spain, in June-July 2011. I would like to thank the organizers of the summer school, and especially Carlos P\'erez, for this opportunity. I would also like to thank Oliver Dragi\v{c}evi\'c for discussions that led me to the observations about the powers of the Beurling operator in Section~\ref{sec:Beurling}. The anonymous referee is warmly thanked for his/her friendly suggestions that eliminated several serious omissions in my original list of citations.

%\bibliography{weighted}
%\bibliographystyle{abbrv}

\end{document}